\documentclass[11pt,a4paper,reqno]{amsart}
\usepackage[foot]{amsaddr}
\usepackage[utf8]{inputenc}
\usepackage[left=1in,right=1in,top=1in,bottom=1in]{geometry}

\usepackage{bbm}
\usepackage{amssymb}
\usepackage{amsmath}
\usepackage{mathtools}
\usepackage{mathdots}
\usepackage{array,multirow,graphicx}
\usepackage{quiver}
\usepackage{float}

\allowdisplaybreaks

\usepackage{amsthm}
\usepackage{enumitem}
\usepackage[hidelinks,unicode,naturalnames,hypertexnames=false]{hyperref}
\usepackage{thmtools}
\usepackage[english,noabbrev,sort,nameinlink,capitalize]{cleveref}


\newcounter{mycounter}[section]

\theoremstyle{plain}
\newtheorem{theorem}[mycounter]{Theorem}
\newtheorem{proposition}[mycounter]{Proposition}
\newtheorem{lemma}[mycounter]{Lemma}
\newtheorem{corollary}[mycounter]{Corollary}

\theoremstyle{definition}
\newtheorem{definition}[mycounter]{Definition}

\theoremstyle{remark}
\newtheorem{remark}[mycounter]{Remark}
\newtheorem{example}[mycounter]{Example}

\numberwithin{equation}{section}
\crefname{subsection}{subsection}{subsections}

\setlist[itemize]{noitemsep, topsep=0pt, leftmargin=1cm}

\newlist{facts}{enumerate}{1}
\setlist[facts, 1]{label=\normalfont\bfseries(\Roman*)}
\newlist{enumalpha}{enumerate}{1}
\setlist[enumalpha, 1]{label=(\alph*)}
\newlist{enumroman}{enumerate}{1}
\setlist[enumroman, 1]{label=(\roman*)}

\newcommand{\leftl}{\mathopen{}\mathclose\bgroup\left}
\newcommand{\rightr}{\aftergroup\egroup\right}
\renewcommand{\bar}{\overline}
\renewcommand{\tilde}{\widetilde}
\renewcommand{\hat}{\widehat}

\renewcommand{\text}{\textnormal}

\DeclareFontFamily{U}{wncy}{}
\DeclareFontShape{U}{wncy}{m}{n}{<->wncyr10}{}
\DeclareSymbolFont{mcy}{U}{wncy}{m}{n}
\DeclareMathSymbol{\Sh}{\mathord}{mcy}{"58}

\newcommand{\wps}[1]{\wp_{#1}}
\newcommand{\wpa}{\wps a}

\newcommand{\bigslant}[2]{{\raisebox{.2em}{$#1$}\left/\raisebox{-.2em}{$#2$}\right.}}

\newcommand{\F}{\mathbb F}

\newcommand{\Z}{\mathbb Z}
\newcommand{\Q}{\mathbb Q}
\newcommand{\R}{\mathbb R}

\newcommand{\cF}{\mathcal F}
\newcommand{\cG}{\mathcal G}
\newcommand{\cI}{\mathcal I}
\newcommand{\cO}{\mathcal O}

\newcommand{\bG}{\mathbb G}
\newcommand{\bP}{\mathbb P}

\newcommand{\fm}{\mathfrak m}

\newcommand{\fS}{\mathfrak S}

\newcommand{\bwp}{\boldsymbol{\wp}}

\newcommand{\dd}{\,\mathrm{d}}

\newcommand{\ab}{\mathrm{ab}}

\newcommand{\ur}{\mathrm{ur}}
\newcommand{\sep}{\mathrm{sep}}

\newcommand{\rel}{\mathrm{rel}}
\newcommand{\surj}{\mathrm{surj}}

\DeclareMathOperator{\Aut}{Aut}

\DeclareMathOperator{\Hom}{Hom}
\DeclareMathOperator{\im}{im}
\DeclareMathOperator{\Gal}{Gal}
\DeclareMathOperator{\GL}{GL}

\DeclareMathOperator{\Stab}{Stab}

\DeclareMathOperator{\lastjump}{lastjump}

\newcommand{\andd}{\quad\quad\text{and}\quad\quad}

\newcommand{\card}[1]{\leftl\lvert#1\rightr\rvert}
\newcommand{\suchthat}[2]{\left\{#1\ \middle\vert\ \begin{matrix}#2\end{matrix}\right\}}
\newcommand{\cardsuchthat}[2]{\card{\suchthat{#1}{#2}}}
\newcommand{\simto}{\overset\sim\to}
\newcommand{\qqquad}{\qquad\quad}
\newcommand{\ppar}[1]{(\!(#1)\!)}
\newcommand{\bbra}[1]{[\![#1]\!]}

\newcommand{\floor}[1]{\leftl\lfloor#1\rightr\rfloor}
\newcommand{\ceil}[1]{\leftl\lceil#1\rightr\rceil}

\DeclareMathOperator{\inv}{inv}

\newcommand{\customref}[2]{\hyperref[#2]{#1~\ref*{#2}}}
\newcommand{\iref}[2]{\customref{\Cref*{#1}}{#2}}

\title[Twisted Artin--Schreier theory I]{Twisted Artin--Schreier theory I: \\ $(C_p \rtimes C_{p-1})$-extensions of local function fields}

\author{Béranger Seguin}
\email{math@beranger-seguin.fr}
\address{Universität Paderborn, Fakultät EIM, Institut für Mathematik, Warburger Str.~100, 33098 Paderborn, Germany.}

\begin{document}

\begin{abstract}
  Over any field~$K$ of characteristic~$p$, we parametrize the $(C_p\rtimes C_{p-1})$-extensions of~$K$ using pairs of elements~$(g,a) \in K \times K^\times$, via a ``twisted'' variant of Artin--Schreier theory.
  Focusing on the case of the local function field~$K := \F_q\ppar\pi$ over a finite field~$\F_q$ of characteristic~$p$, we use this parametrization to obtain an exact formula for the count of all $(C_p\rtimes C_{p-1})$-extensions of~$K$ with a given discriminant.
\end{abstract}

\maketitle

{
  \setcounter{tocdepth}{1}
  \hypersetup{linkcolor=black}
  \tableofcontents{}
}

In the whole paper, we fix a prime number~$p$.
The cyclic group of order~$n$ is denoted by $C_n := \Z/n\Z$.

\section{Introduction}

The cyclic group~$C_{p-1} \simeq \F_p^\times$ acts on the cyclic group~$C_p \simeq \F_p$ by scalar multiplication.
Let~$C_p \rtimes C_{p-1}$ be the semi-direct product associated to this action (the one-dimensional affine group over~$\F_p$), and let $\pi \colon C_p \rtimes C_{p-1} \twoheadrightarrow C_{p-1}$ be the canonical surjection.
In this paper, we are interested in the description of all $(C_p \rtimes C_{p-1})$-extensions of a fixed field~$K$ of characteristic~$p$, and specifically in the distribution of their discriminants when~$K$ is a local function field over a finite field of characteristic~$p$.
For us, the word ``extension'' is understood in a broad sense, cf.~\Cref{def:extension}, but it in particular includes Galois field extensions having~$C_p \rtimes C_{p-1}$ as their Galois group.

\subsection{Context and previous work}

For a fixed finite group~$G$ and a fixed field~$K$, the problem of counting $G$-extensions of~$K$ ordered by discriminant has a long history:
\begin{itemize}
  \item
    When~$K$ is either a number field or a global function field and~$G$ is an abelian group whose order is not divisible by the characteristic of~$K$, this problem was solved using class field theoretic methods in \cite{maki,wright,wood-probabilities-of-local-behaviors}.
  \item
    The case of non-abelian groups~$G$ over a number field~$K$ is the subject of a well-studied conjecture of Malle, see for example \cite{malle1,inductive,loughransantens}.
  \item
    When~$K = \F_q(T)$ is a rational function field over a finite field whose characteristic does not divide~$\card G$ (the ``tame case''), an approach via the cohomology of moduli spaces of tamely ramified covers of~$\bP^1$ (Hurwitz spaces) has culminated in an essentially complete solution for large~$q$, cf.~\cite{etw,landesmanlevy,santenshur}.
    This approach also yields results on the distribution of the coprime-to-$p$ part of class groups of function fields of characteristic~$p$, cf.~\cite{evw,landesmanlevy2}, \cite[Theorem~1.3.2]{landesmanlevy}, and \cite{woodsawin}.
  \item
    When~$K$ is a (local or global) function field of characteristic~$p$ and~$G$ is a $p$-group (the ``purely wild case''), explicit parametrizations based on generalizations of Artin--Schreier theory together with class field theoretic tools have led to recent progress, cf.~\cite{lagemann1,lagemann2,potthast,gundab,wildcount,wildlocglob}.
\end{itemize}

\medskip

The motivation behind this paper is to study an example in the ``missing case'' over function fields of characteristic~$p$, in which both tame and wild ramification can occur, i.e., the case where~$G$ is neither a $p$-group nor a group of order coprime to~$p$.
The group~$C_p \rtimes C_{p-1}$ is still nice in several ways: it has a unique $p$-Sylow subgroup (a property that it shares with both the purely wild and the tame case), and it is solvable.
It is therefore a natural first step in the direction of general groups.

\medskip

Although this paper discusses local function fields, asymptotic results over global (in fact rational) function fields will follow in a sequel paper,%
\footnote{
  That is, if no AI--fanatic scoops me to it.
}
using much of the results obtained here.

\medskip

Previous work about similar local counting questions includes \cite[Theorems~3.17 and 3.18]{muller-thesis}, \cite[Theorem~1.1]{klumul-frob}, and \cite[Theorem~3.2]{tavyas}.
Klüners and Müller count by discriminant and give asymptotic counts, whereas Tavernier and Yasuda count by ramification jump and give exact counts.
In this paper, we give exact counts for the $(C_p \rtimes C_{p-1})$-extensions of local function fields of characteristic~$p$ when counting by discriminant (\Cref{thm:intro:local-total-count}).

\medskip

It should also be mentioned that moduli spaces of wildly ramified covers of curves with similar Galois groups have been studied, notably in \cite{obuspries,zhang,modform} and in \cite[Subsection~3.2]{tavyas}.
The geometric description of these moduli spaces could be relevant for counting questions, in the spirit of the geometric approach to the tame case from \cite{etw,landesmanlevy}.

\subsection{Twisted Artin--Schreier theory.}
In \Cref{thm:main-param}, we describe a bijective parametrization of all the $(C_p \rtimes C_{p-1})$-extensions of any field~$K$ of characteristic~$p$.
Compared to inductive approachs in which one first constructs a $C_{p-1}$-extension~$K(\alpha)|K$ and then uses Artin--Schreier theory over~$K(\alpha)$ to parametrize the $C_p$-extensions of~$K(\alpha)$ (cf.~\cite{klumul-frob}), this direct description shortcuts the need to determine which extensions have the desired Galois group.
Moreover, extensions are directly parametrized by elements of the ground field~$K$.

By Kummer theory, any $C_{p-1}$-extension of~$K$ is induced by some element~$a \in K^\times$, cf.~\Cref{subsn:kummer}.
Once such an element~$a \in K^\times$ is fixed, our parametrization of the $(C_p \rtimes C_{p-1})$-extensions lying above the $C_{p-1}$-extension associated to~$a$ relies on a twisted variant of Artin--Schreier theory, in which the Artin--Schreier map $\wp \colon x \mapsto x^p - x$ is replaced by the ($\F_p$-linear) \emph{twisted Artin--Schreier map}
\[
  \wpa \colon x \longmapsto x^p - ax.
\]
A large portion of this article is dedicated to the study of the map~$\wpa$, of its image, and of the quotient~$K/\wpa(K)$ when~$K$ is a local function field~$\kappa\ppar\pi$ with finite (or at least perfect) residue field~$\kappa$ (\Cref{sn:local,sn:local-count}).

The importance of the quotient $\F_p$-vector space~$K/\wpa(K)$ is that it parametrizes the $(C_p \rtimes C_{p-1})$-extensions of~$K$ lying above the $C_{p-1}$-extension of~$K$ associated to~$a$.
Specifically, if~$a \in K^\times$ and~$g \in K$ are fixed, if $\alpha \in (K^\sep)^\times$ is a $(p-1)$-th root of~$a$, and if $\gamma \in \wpa^{-1}(g)$ (i.e., if~$\gamma \in K^\sep$ is a solution to the equation $\gamma^p - a\gamma = g$), then we obtain a continuous group homomorphism
\[
  \begin{matrix}
    \rho \colon &
    \Gamma_K &
    \longrightarrow &
    C_p \rtimes C_{p-1} \\
    &
    \tau &
    \longmapsto &
    \Bigl(
      \frac{\tau(\gamma)-\gamma}\alpha, \;
      \frac{\tau(\alpha)}\alpha
    \Bigr)
  \end{matrix}
\]
and all continuous group homomorphisms~$\rho \in \Hom(\Gamma_K, C_p \rtimes C_{p-1})$ are obtained in this way for some choice of $a$, $g$, $\alpha$ and~$\gamma$.
In \Cref{thm:main-param} and in \Cref{subsn:fiberwise}, we describe the exact conditions under which different choices of~$a$ and~$g$ (and then of~$\alpha$ and~$\gamma$) lead to the same~$\rho$ (either exactly, or up to conjugacy).

\medskip

Equivalently, letting~$\chi_a \in \Hom(\Gamma_K, \F_p^\times)$ be the Kummer character associated to~$a \in K^\times$ (\Cref{def:kummer-character}), the elements of~$K/\wpa(K)$ correspond to the elements of the cohomology group~$H^1(\Gamma_K, \F_p)$ for the non-trivial action~$\tau.\mu = \chi_a(\tau)\mu$ of $\Gamma_K$ on~$\F_p$, cf.~\Cref{eq:twisted-artsch}.
Hence, the problem of describing the asymptotic distribution of these cohomology classes fits into the landscape of the ``twisted Malle conjecture'' for the first Galois cohomology group of number fields introduced by Alberts in \cite{alberts} (see also \cite{alberts2,inductive}).
What we are studying is a function field analogue of that problem, in a situation with wild ramification.

\subsection{Main result}
When~$K$ is a local function field, for each $a \in K^\times$, we describe in \Cref{thm:local-fundom} distinguished choices of representatives in~$K$ of each coset in~$K/\wpa(K)$, analogous to the fundamental domain used in ordinary Artin--Schreier theory, cf.~\cite[Proposition~5.2(b)]{potthast}.
When an element~$g \in K$ is such a distinguished representative of its coset (it is ``reduced''), the discriminant of the $(C_p \rtimes C_{p-1})$-extension of~$K$ associated to the pair~$(g,a)$ is computed in \Cref{thm:local-disc}.
By counting distinguished representatives, we are then able to give exact counts for the number of $(C_p \rtimes C_{p-1})$-extensions of the local function field $K = \F_q\ppar\pi$, ordered by their discriminant exponent~$\delta$ (defined in \Cref{def:disc}):

\begin{theorem}[\Cref{thm:local-total-count}]
  \label{thm:intro:local-total-count}
  Let $K = \F_q\ppar\pi$.
  For any $x \in \R_{\geq0}$ and any $u \in \{0,\ldots,p-2\}$, define the following numbers (cf.~\Cref{def:eta-Nux,lem:Nu})
  \begin{align*}
    \eta_u(x)
    & :=
    \frac{x - p(p-1) + \gcd(p-1, \, u)}{(p-1)^2}
    - \frac {pu}{p-1}
    \\
    N_u(x)
    & :=
    \begin{cases}
      - \infty & \text{if } x < p(p-1) - p \gcd(p-1, \, u) \\
      0 & \text{if } p(p-1) - p\gcd(p-1, \, u) \leq x \leq p(p-1)-\gcd(p-1, \, u) \\
      u+\floor{\eta_u(x)}-\floor{\frac{\eta_u(x)}p} & \text{if } x \geq p(p-1)-\gcd(p-1, \, u).
    \end{cases}
  \end{align*}
  Then, for any $x \in \R_{\geq0}$, we have the total count
  \[
    \frac{1}{p(p-1)}
    \cardsuchthat{
      \rho \in \Hom(\Gamma_K, C_p \rtimes C_{p-1})
    }{
      \delta(\rho) \leq x
    }
    =
    \sum_{u = 0}^{p-2}\,
      q^{N_u(x)}
  \]
  and, for $p \neq 2$, we have the following counts for surjective (resp.~dihedral) extensions:
  \begin{align*}
    \frac{1}{p(p-1)}
    \cardsuchthat{
      \rho \in \Hom^\surj(\Gamma_K, C_p \rtimes C_{p-1})
    }{
      \delta(\rho) \leq x
    }
    & =
    \sum_{u = 0}^{p-2}\,
      \frac{\varphi(\gcd(u,p-1))}{\gcd(u,p-1)}
      \Bigl(
        q^{N_u(x)}
        - 1
      \Bigr)
    \\
    \frac{1}{2p}
    \cardsuchthat{
      \rho \in \Hom^\surj(\Gamma_K, D_p)
    }{
      \delta(\rho) \leq x
    }
    & =
    \frac12
    q^{N_0(\frac{p-1}2x)}
    +
    q^{N_{\frac{p-1}2}(\frac{p-1}2x)}
    -\frac32
  \end{align*}
  where~$\varphi$ denotes Euler's totient function.
\end{theorem}

We refer to \Cref{rk:count-hom} for an explanation of why we count extensions in this way.
See also \Cref{prop:local-fiberwise-count} for a fiberwise version of the local counting result.

A consequence of \Cref{thm:intro:local-total-count} is that for fixed~$p$ the number of $(C_p\rtimes C_{p-1})$-extensions of~$\F_q\ppar\pi$ with bounded discriminant is a polynomial with coefficients in $\Z_{\geq0}$ in the size~$q$ of the residue field.
The values of the exponents~$N_u(x)$ are discussed further in \Cref{subsn:behavior-nux}: the key takeaway is that~$N_u(x) \approx \frac{x}{p(p-1)}$.
The concrete example of $\fS_3$-extensions in characteristic~$3$ is detailed in \Cref{ex:s3-char3}, giving the following exact formula when~$q$ is a power of~$3$:
\[
  \frac{1}{6}
  \cardsuchthat{
    \rho \in \Hom(\Gamma_{\F_q\ppar\pi}, \fS_3)
  }{
    \delta(\rho) = x
  }
  =
  \begin{cases}
    1
    & \text{if } x = 0 \text{ or } x = 3
    \\
    0
    & \text{if } (x \bmod 12) \in \{1,2,4,5,6,9,10,11\}
    \\
    q^{\floor{\frac x6}} - 
    q^{\floor{\frac x6}-1}
    & \text{if }
    x \geq 6 \text{ and }
    (x \bmod 12) \in \{0,3,7,8\}.
  \end{cases}
\]

\subsection{Acknowledgements}

This work was supported by the Deutsche Forschungsgemeinschaft (DFG, German Research Foundation)---Project-ID 491392403---TRR 358 (Project A4).
I thank Jürgen Klüners and Fabian Gundlach for helpful discussions.

No AI was used at any point in the process of preparing this paper, neither for thinking nor for writing.
(This paper was mostly written in February 2026, where this would have been obvious. Eight months later, how rare this seems to have become!)
All mistakes (mathematical, grammatical, or otherwise) are fully mine, and hence purely human.

\section{Preliminaries}

\subsection{Definitions}

For any field~$K$, let~$K^\sep$ be a separable closure of~$K$, and let~$\Gamma_K := \Gal(K^\sep|K)$ be the absolute Galois group of~$K$ equipped with the Krull topology.
All finite groups are equipped with the discrete topology.
For any two topological groups~$\Gamma$, $G$, the notation~$\Hom(\Gamma, G)$ refers to the pointed set of \emph{continuous} group homomorphisms~$\Gamma \to G$, and~$\Hom^\surj(\Gamma,G)$ is the subset formed of surjective continuous group homomorphisms~$\Gamma \twoheadrightarrow G$.
Similarly, if~$\Gamma$ acts on~$G$ via group automorphisms, the pointed sets~$H^i(\Gamma, G)$ (if~$G$ is abelian or if~$i \in \{0,1\}$) are meant in the sense of \emph{continuous} group cohomology, and are abelian groups when~$G$ is abelian.


\subsubsection{Extensions}

Let~$G$ be a finite group, and let~$K$ be a field.
We use the following notion of extensions:

\begin{definition}
  \label{def:extension}
  A \emph{$G$-extension of~$K$} is an orbit~$[\rho]$ of continuous group homomorphisms~$\rho \in \Hom(\Gamma_K, G)$ under the conjugation action of~$G$, i.e., an element of the pointed set~$H^1(\Gamma_K, G)$  for the trivial $\Gamma_K$-action on~$G$.
\end{definition}

If~$G$ is abelian, $G$-extensions of~$K$ are simply continuous group homomorphisms from~$\Gamma_K$ (or from~$\Gamma_K^\ab$) to~$G$.

There is a correspondence between $G$-extensions of~$K$ and specific $\card G$-dimensional étale $K$-algebras with an action of~$G$, cf.~\cite[Lemma~2.1]{wildcount}.
Those include the Galois field extensions~$L|K$ with an isomorphism $\Gal(L|K) \simeq G$ (the case of surjective group homomorphisms~$\rho \in \Hom^\surj(\Gamma_K, G)$), but also the étale $K$-algebras~$L^{[G:H]}$ with $\Gal(L|K) \simeq H$ for a subgroup~$H \subseteq G$ (the case $\im\rho \simeq H$).
For instance, the trivial homomorphism corresponds to the étale $K$-algebra $K^{\card G}$ (the ring of maps~$G \to K$) on which~$G$ acts by permuting the factors.

\begin{remark}
  \label{rk:count-hom}
  Let~$\Sigma \subseteq H^1(\Gamma_K, G)$ be finite.
  The ``correct'' way to count $G$-extensions of~$K$ in~$\Sigma$ is to consider the \emph{groupoid cardinality}, i.e., to weight each extension~$[\rho]$ by the inverse of the size of its automorphism group (the automorphism group of~$[\rho]$ is the automorphism group of the corresponding étale $K$-algebra, which coincides with the subgroup of elements of~$G$ acting trivially on~$\rho$ by conjugation, i.e., the centralizer of~$\im\rho$ in~$G$, cf.~\cite[Lemma~2.1]{wildcount}).
  By the orbit--stabilizer theorem, we have
  \begin{align*}
    \sum_{[\rho] \in \Sigma}
      \frac1{\card{\Aut [\rho]}}
    =
    \frac1{\card G}
    \sum_{[\rho] \in \Sigma}
      \frac
        {\card G}
        {\card{\Stab_G(\rho)}}
    & =
    \frac1{\card G}
    \sum_{[\rho] \in \Sigma}
      \cardsuchthat
        {\rho' \in \Hom(\Gamma_K, G)}
        {[\rho']=[\rho]}
    \\ & =
    \frac1{\card G}
    \cardsuchthat
      {\rho \in \Hom(\Gamma_K, G)}
      {[\rho] \in \Sigma}.
  \end{align*}
\end{remark}

\subsubsection{Inertial types}

Let~$K$ be either a local or a global field, and let~$K^\ur \subseteq K^\sep$ be the maximal unramified extension of~$K$.
(If~$K$ is a local field, then~$\Gamma_{K^\ur}$ is the inertia subgroup of~$\Gamma_K$.)
Let~$G$ be a topological group on which~$\Gamma_K$ acts trivially.

\begin{definition}
  \label{def:unramified}
  A $G$-extension $[\rho] \in H^1(\Gamma_K, G)$ (resp.~a homomorphism $\rho \in \Hom(\Gamma_K, G)$) is \emph{unramified} if~$\rho(\Gamma_{K^\ur})=1$.
  We let~$H^1_\ur(\Gamma_K, G) \subseteq H^1(\Gamma_K, G)$ and~$\Hom_\ur(\Gamma_K, G) \subseteq \Hom(\Gamma_K, G)$ be the pointed subsets consisting of unramified extensions/homomorphisms.
\end{definition}

We have $H^1_\ur(\Gamma_K, G) = H^1(\Gal(K^\ur|K), G)$ and $\Hom_\ur(\Gamma_K, G) = \Hom(\Gal(K^\ur|K), G)$.

\begin{definition}
  \label{def:inertial-type}
  For any $G$-extension $[\rho] \in H^1(\Gamma_K, G)$ of~$K$, the \emph{inertial type of $[\rho]$} is the element $[\rho|_{\Gamma_{K^\ur}}] \in H^1(\Gamma_{K^\ur}, G)$ obtained via restriction to~$\Gamma_{K^\ur}$.
  We denote by~$\cI(K, G)$ the pointed subset of~$H^1(\Gamma_{K^\ur}, G)$ obtained as the image of~$H^1(\Gamma_K, G)$ under the restriction map.
\end{definition}

When~$G$ is abelian, $H^1(\Gamma_K, G) = \Hom(\Gamma_K, G)$ is an abelian group, and
\[
  \cI(K, G)
  \simeq
  \bigslant
    {\Hom(\Gamma_K, G)}
    {\Hom_\ur(\Gamma_K, G).}
\]

\subsubsection{Local last jumps and discriminant exponents}

Let~$K$ be a local field with perfect residue field.
For any real number $v \geq -1$, let $\Gamma_K^v$ be the $v$-th ramification subgroup in the upper numbering, cf.~\cite[Chap.~IV, §3]{serrecl}.
These higher ramification subgroups form a non-increasing filtration of~$\Gamma_K$ by normal subgroups.
Here are some of their properties:
\begin{itemize}
  \item
    $\Gamma_K^{-1}$ is all of~$\Gamma_K$
  \item
    for any $v \in (-1,0]$, $\Gamma_K^v = \Gamma_K^0$ is the inertia subgroup
  \item
    for any $v \in (0,1]$, $\Gamma_K^v = \Gamma_K^1$ is the wild inertia subgroup (the unique pro-$p$-Sylow subgroup of~$\Gamma_K^0$)
  \item
    $\bigcap_{v \geq -1} \Gamma_K^v = 1$.
\end{itemize}

\begin{definition}
  \label{def:disc}
  Let~$G$ be a finite group, and let $\rho \in \Hom(\Gamma_K, G)$.
  The \emph{last jump} of~$\rho$ is the nonnegative rational number
  \[
    \lastjump\rho :=
    \inf\suchthat{
        v \in \Q_{\geq 0}
    }{
        \rho(\Gamma_K^v) = 1
    }
    \in \frac1{\card G}\Z_{\geq0}.
  \]
  and the \emph{discriminant exponent} of~$\rho$ is the nonnegative integer
  \[
    \delta(\rho)
    :=
    \card{G}
    \int_{-1}^{+\infty}
      \leftl(
        1
        -
        \frac1{\card {\rho(\Gamma_K^v)}}
      \rightr)
      \dd v
      \quad
      \in \Z_{\geq 0}.
  \]
\end{definition}

Both the last jump and the discriminant exponent depend on~$\rho$ only up to $G$-conjugation, and depend only on the restriction of~$\rho$ to the inertia subgroup~$\Gamma_K^0$ (we have $\Gamma_K^v=\Gamma_K^0$ for all $-1<v\leq0$).
Hence, these invariants are well-defined at the level of $G$-extensions $[\rho] \in H^1(\Gamma_K, G)$ (cf.~\Cref{def:extension}), and at the level of inertial types $[\rho] \in \cI(K,G)$ (cf.~\Cref{def:inertial-type}).

The last jump of~$\rho$ vanishes if and only if~$\rho$ is at most tamely ramified (i.e., $\rho(\Gamma_K^1)=1$), and the discriminant exponent vanishes if and only if~$\rho$ is unramified (i.e., $\rho(\Gamma_K^0)=1$, cf.~\Cref{def:unramified}).
If~$\rho$ is at most tamely ramified, then $\Gamma_K^v = 1$ for all $v > 0$, and $\Gamma_K^v = \Gamma_K^0$ for all $v \in (-1,0]$, so $\delta(\rho) = \card G (1-\frac1e)$ where $e=\card{\rho(\Gamma_K^0)}$.

The following proposition relates~$\delta$ to relative discriminants in the usual sense:

\begin{proposition}
  \label{prop:discrexp-is-discrexp}
  Let~$L|K$ be a finite Galois extension of local fields, and let $G := \Gal(L|K)$.
  Regard~$L$ as embedded in~$K^\sep$, and let $\rho \in \Hom^\surj(\Gamma_K, G)$ be the (surjective) continuous homomorphism corresponding to the action of~$\Gamma_K$ on~$L$.
  Let $\mathfrak d_{L|K}$ be the relative discriminant of~$L|K$ (a nonzero ideal of~$\cO_K$, cf.~\cite[Chap.~III, §3]{serrecl}).
  Then, $v_K(\mathfrak d_{L|K}) = \delta(\rho)$.
\end{proposition}

\begin{proof}
  By \cite[Chap.~III, §3, Prop.~6]{serrecl}, we have $\mathfrak d_{L|K} = \mathrm{N}_{L|K}(\mathfrak D_{L|K})$ where~$\mathfrak D_{L|K} \subseteq \cO_L$ is the different of the extension.
  Hence:
  \[
    v_K(\mathfrak d_{L|K}) =
    v_K(\mathrm{N}_{L|K}(\mathfrak D_{L|K})) =
    \frac{1}{e_{L|K}}
    v_L(\mathrm{N}_{L|K}(\mathfrak D_{L|K}))
    =
    \frac{[L:K]}{e_{L|K}}
    v_L(\mathfrak D_{L|K})
    =
    \frac{\card G}{e_{L|K}}
    v_L(\mathfrak D_{L|K}).
  \]
  Moreover:
  \begin{align*}
    v_L(\mathfrak D_{L|K})
    &=
    \sum_{i = 0}^{+\infty}
      \Bigl(
        \card{G_i} - 1
      \Bigr)
    \tag*{by \cite[Chap.~IV, §2, Prop.~4]{serrecl}}
    \\
    &=
    \int_{-1}^{+\infty}
      \Bigl(
        \card{G_t} - 1
      \Bigr)
      \dd t
    \tag*{as $G_t = G_{\ceil t}$}
    \\
    &=
    \int_{-1}^{+\infty}
      \Bigl(
        \card{G^s} - 1
      \Bigr)
      [G^0:G^s]
      \dd s
    \tag*{with $s = \varphi_{L|K}(t)$, cf.~\cite[Chap.~IV, §3]{serrecl}}
    \\
    &=
    \card{G^0}
    \int_{-1}^{+\infty}
      \Bigl(
        1 - \frac{1}{\card{G^s}}
      \Bigr)
      \dd s
    =
    e_{L|K} \cdot
    \int_{-1}^{+\infty}
      \Bigl(
        1 - \frac{1}{\card{\rho(\Gamma_K^s)}}
      \Bigr)
      \dd s.
      \tag*\qedhere
  \end{align*}
\end{proof}

\begin{remark}
  \label{rk:discr-étale-alg}
  If~$L|K$ is the étale $K$-algebra corresponding to some $\rho \in \Hom(\Gamma_K, G)$ (not necessarily surjective), then~$L$ is isomorphic to the product of~$[G:\im\rho]$ copies of the Galois extension $L' := (K^\sep)^{\ker \rho}$ of~$K$, of Galois group~$\im\rho$, so $v_K(\mathfrak d_{L|K}) = [G:\im\rho] \cdot v_K(\mathfrak d_{L'|K})$ in the sense of discriminants of étale algebras.
  If we let $\tilde\rho \in \Hom^\surj(\Gamma_K, \im\rho)$ be the corestriction of~$\rho$, i.e., the surjective homomorphism associated to~$L'$, we have by definition $\delta(\rho) = [G:\im\rho]\delta(\tilde\rho)$, and $\delta(\tilde\rho) = v_K(\mathfrak d_{L'|K})$ by \Cref{prop:discrexp-is-discrexp}.
  Therefore, we also have $\delta(\rho) = v_K(\mathfrak d_{L|K})$.
\end{remark}

\begin{definition}
  \label{def:disc-rel}
  Consider a short exact sequence $1 \to N \to G \overset\pi\to Q \to 1$ of finite groups, and let $\rho \in \Hom(\Gamma_K, G)$.
  The \emph{relative discriminant exponent of~$\rho$ (with respect to~$\pi$)} is defined as
  \[
    \delta^\rel(\rho) :=
    \frac{1}{\card Q}
    \Bigl(
      \delta(\rho) - \card N \cdot \delta(\pi\circ\rho)
    \Bigr)
    =
    \card N
    \int_{-1}^{+\infty}
    \leftl(
      \frac1{\card{\pi(\rho(\Gamma_K^s))}}
      -
      \frac1{\card{\rho(\Gamma_K^s)}}
    \rightr)
    \dd s
  \]
  so that $\delta(\rho) = \card N \cdot \delta(\pi\circ\rho) + \card Q \cdot \delta^\rel(\rho)$.
\end{definition}

\subsection{Reminders on Kummer theory}
\label{subsn:kummer}

Fix a field~$K$ of characteristic~$p$.
We shall focus on the $C_{p-1}$-extensions of~$K$.
(Any $C_d$-extension with~$d \mid p-1$ is in particular a $C_{p-1}$-extension.)
Observe that the isomorphism $C_{p-1} \simeq \F_p^\times$ holds even as $\Z[\Gamma_K]$-modules: indeed, the Galois action is trivial as~$\F_p^\times \subseteq K$.
A description of $\Hom(\Gamma_K, C_{p-1}) = H^1(\Gamma_K, \F_p^\times)$ is obtained by considering the short exact sequence
\[
  1
  \longrightarrow
  \F_p^\times
  \overset\subset\longrightarrow
  (K^\sep)^\times
  \overset{\bullet^{p-1}}\longrightarrow
  (K^\sep)^\times
  \longrightarrow
  1
\]
which leads to the following exact sequence in Galois cohomology, in which $H^1(\Gamma_K, (K^\sep)^\times) = 1$ by Hilbert's Theorem 90 \cite[Chap.~X, §1, Prop.~2]{serrecl}:
\[
  1
  \longrightarrow
  \F_p^\times
  \overset\subset\longrightarrow
  K^\times
  \overset{\bullet^{p-1}}\longrightarrow
  K^\times
  \longrightarrow
  H^1(\Gamma_K, \F_p^\times)
  \longrightarrow
  1
\]
whence an isomorphism of abelian groups (``Kummer theory'')
\begin{equation}
  \label{eq:kummer}
  \Hom(\Gamma_K, C_{p-1})
  \simeq
  K^\times/(K^\times)^{p-1}.
\end{equation}

(Here, $(K^\times)^{p-1}$ is the subgroup of $(p-1)$-th powers in~$K^\times$.
This is not a Cartesian product!)

\begin{remark}
  \label{rk:kummer-ext-finite-fields}
  For any finite field~$\F$, the short exact sequence $1 \to \F_p^\times \overset\subset\to \F^\times \overset{\bullet^{p-1}}\to (\F^\times)^{p-1} \to 1$ shows that $\card {\F^\times} = \card {\F_p^\times} \cdot \card{(\F^\times)^{p-1}}$, whence
  \[
    \card{\F^\times/(\F^\times)^{p-1}}
    =
    \frac{\card{\F^\times}}{\card{(\F^\times)^{p-1}}}
    =
    \card{\F_p^\times}
    =
    p-1.
  \]
\end{remark}

\begin{definition}
  \label{def:kummer-character}
  For any~$a \in K^\times$, we let $\chi_a \in \Hom(\Gamma_K, C_{p-1})$ be the Kummer character associated to~$a$, given by $\chi_a(\tau) = \tau(\alpha)/\alpha$ where~$\alpha \in (K^\sep)^\times$ is any $(p-1)$-th root of~$a$.
\end{definition}

\begin{lemma}
  \label{lem:kummer-multiplicative}
  For any $a,b \in K^\times$, we have $\chi_a + \chi_b = \chi_{ab}$.
\end{lemma}

\begin{proof}
  This follows from the fact that $a \mapsto \chi_a$ is a group isomorphism $K^\times/(K^\times)^{p-1} \simto \Hom(\Gamma_K, C_{p-1})$, cf.~\Cref{eq:kummer}.
\end{proof}

\begin{definition}
  For any valuation~$v$ of~$K$ and for any $\chi \in \Hom(\Gamma_K, C_{p-1})$, we define
  \[
    \inv_v(\chi)
    :=
    v(a) \bmod {(p-1)}
    \quad
    \in C_{p-1}
  \]
  where~$a \in K^\times$ is any element such that $\chi = \chi_a$.
  (This depends only on~$[a] \in K^\times/(K^\times)^{p-1}$.)
\end{definition}

\subsubsection{Local Kummer extensions}

Assume that $K = \kappa\ppar\pi$ is a local function field with finite residue field~$\kappa$ of characteristic~$p$, and let~$v$ be the $\pi$-adic valuation.

\begin{lemma}
  \label{lem:power-criterion-leadcoeff}
  Let $a \in K^\times$, let $u := v(a)$, and let $A \in \kappa^\times$ be the coefficient of~$a$ in front of~$\pi^u$.
  Then:
  \[
    a \in (K^\times)^{p-1} 
    \iff
    \Bigl(
      p-1 \mid u
      \text{ and }
      A \in (\kappa^\times)^{p-1}
    \Bigr).
  \]
\end{lemma}

\begin{proof}
  Use Hensel's lemma, or define an explicit $(p-1)$-th root using the power series for $(1+x)^{\frac1{p-1}}$.
\end{proof}

\begin{proposition}
  \label{prop:kummer-local-normal-form}
  The following map is a well-defined group isomorphism:
  \[
    \begin{matrix}
      \Xi \colon & 
      C_{p-1} \times \kappa^\times/(\kappa^\times)^{p-1} & \longrightarrow & K^\times/(K^\times)^{p-1} \\
      & (\bar u, [\tilde a]) & \longmapsto & [\pi^u \tilde a].
    \end{matrix}
  \]
\end{proposition}

\begin{proof}
  If $(\bar u, [\tilde a]) = (\bar{u'}, [\tilde a'])$ in $C_{p-1} \times \kappa^\times/(\kappa^\times)^{p-1}$, then $u' = u+k(p-1)$ for some $k \in \Z$, and $\tilde a' = \mu^{p-1} \tilde a$ for some $\mu \in \kappa^\times$.
  Then, $\pi^{u'}\tilde a' = (\pi^k\mu)^{p-1} \pi^u \tilde a$ has the same image as~$\pi^u\tilde a$ in~$K^\times/(K^\times)^{p-1}$.
  Therefore, $\Xi$ is well-defined.
  It is also clearly a group homomorphism.

  We now prove surjectivity.
  Any~$a \in K^\times$ can be written as~$\pi^u \tilde a a'$ with $u=v(a)$, $\tilde a \in \kappa^\times$ and $a' \in 1+\fm$.
  Hence, it suffices to show that $1+\fm \subseteq (K^\times)^{p-1}$, which follows from \Cref{lem:power-criterion-leadcoeff}.

  To prove injectivity, assume that $\pi^u \tilde a$ is a $(p-1)$-th power.
  Then, $p-1 \mid u$, so we can simply assume that $u=0$.
  The element $\tilde a \in \kappa^\times$ must be a $(p-1)$-th power in~$K^\times$, so its $(p-1)$-th roots must belong to $K^\times \cap \bar\kappa^\times = \kappa^\times$, so in fact $\tilde a \in (\kappa^\times)^{p-1}$ and then $[\tilde a] = [1]$ in $\kappa^\times/(\kappa^\times)^{p-1}$.
\end{proof}

\begin{corollary}
  \label{cor:kummer-local-inertype}
  The map~$\inv_v$ induces an isomorphism $\cI(K, C_{p-1}) \simeq C_{p-1}$.
\end{corollary}

\begin{proof}
  By Kummer theory, we have $\Hom(\Gamma_K, C_{p-1}) \simeq K^\times/(K^\times)^{p-1}$ and  $\Hom_\ur(\Gamma_K, C_{p-1}) \simeq \Hom(\Gamma_{\kappa}, C_{p-1}) \simeq \kappa^\times/(\kappa^\times)^{p-1}$.
  The result follows by combining these isomorphisms with \Cref{prop:kummer-local-normal-form}.
  (We have $\inv_v(\chi_{\pi^u \tilde a}) = u \pmod {p-1}$.)
\end{proof}

Let~$\tilde\sigma \in \Gamma_K$ be a lift of the Frobenius automorphism $\sigma \colon x \mapsto x^{\card\kappa}$ of~$\bar\kappa$ (a topological generator of~$\Gal(\bar\kappa|\kappa)$).

\begin{proposition}
  \label{prop:kummer-inertia}
  Fix $u \in C_{p-1}$.
  There is a bijection
  \[
    \begin{matrix}
    \suchthat{\chi \in \Hom(\Gamma_K, C_{p-1})}{\inv \chi = u} &
    \overset\sim\longleftrightarrow &
    C_{p-1} \\
    \chi &
    \longmapsto &
    \chi(\tilde\sigma).
    \end{matrix}
  \]
  For any $\chi \in \Hom(\Gamma_K, C_{p-1})$ with $\inv\chi = u$, we have the following descriptions (as subgroups of~$C_{p-1}$)
  \[
    \chi(\Gamma_K) = \langle u, \chi(\tilde\sigma) \rangle
    \andd
    \chi(\Gamma_K^0) = \langle u \rangle.
  \]
  In particular:
  \[
    \card{\chi(\Gamma_K)}
    =
    \frac{p-1}{\gcd(u,\chi(\tilde\sigma),p-1)}
    \andd
    \card{\chi(\Gamma_K^0)}
    =
    \frac{p-1}{\gcd(u,p-1)}.
  \]
\end{proposition}

\begin{proof}
  By \Cref{prop:kummer-local-normal-form}, \Cref{eq:kummer} and \Cref{lem:kummer-multiplicative}, we have
  \[
    \suchthat{
      \chi \in \Hom(\Gamma_K, C_{p-1})
    }{
      \inv \chi = u
    }
    =
    \suchthat{
      \chi_{\pi^u \tilde a}
    }{
      \tilde a \in \kappa^\times
    }
    =
    \Hom_\ur(\Gamma_K, C_{p-1})
    +
    \chi_{\pi^u}
  \]
  so a homomorphism~$\chi$ with $\inv\chi = u$ is equivalently given by the corresponding unramified homomorphism~$\chi - \chi_{\pi^u} = \chi_{\tilde a} \in \Hom_\ur(\Gamma_K, C_{p-1})$.
  Since $\Hom_\ur(\Gamma_K, C_{p-1}) \simeq \Hom(\Gamma_{\F_q}, C_{p-1})$ and~$\Gamma_{\F_q}$ is topologically generated by~$\sigma$ (the image of~$\tilde\sigma$), $\chi$ is equivalently given by the element $\chi_{\tilde a}(\sigma) = \chi_{\tilde a}(\tilde\sigma) = \chi(\tilde\sigma) - \chi_{\pi^u}(\tilde\sigma)$, whence the first bijection.
  Any $\tau \in \Gamma_K$ decomposes uniquely as $\tilde\sigma^r \tau_0$ for some $\tau_0 \in \Gamma_K^0$ and $r \in \hat\Z$ ($r$ is the image of~$\tau$ in $\Gamma_{\F_q} \simeq \hat\Z$, and $\tau_0 := \tilde\sigma^{-r} \tau$ is in~$\Gamma_K^0$ as its image in $\Gamma_{\F_q}$ is trivial by definition).
  Therefore, $\chi(\Gamma_K)$ is topologically generated (and hence generated, as~$C_{p-1}$ is finite) by its subgroup~$\chi(\Gamma_K^0)$ and by~$\chi(\tilde\sigma)$.
  It remains only to show that~$\chi(\Gamma_K^0)$ is generated by~$u$.
  Since $\chi - \chi_{\pi^u} \in \Hom_\ur(\Gamma_K, C_{p-1})$, $\chi$ coincides with~$\chi_{\pi^u}$ in restriction to~$\Gamma_K^0$, so we can focus on~$\chi_{\pi^u}$.
  Let $g := \gcd(u, p-1)$, $d = (p-1)/g$, and $u' = u/g$.
  The integers $d$ and $u'$ are coprime, so we fix $x,y \in \Z$ such that $xd+yu' = 1$.
  Pick a $d$-th root~$\alpha$ of $\pi^{u'}$, which is then also a $(p-1)$-th root of $\pi^u$.
  Since subgroups of~$C_{p-1}$ are determined by their sizes, it suffices to prove that~$K(\alpha)|K$ has ramification index equal to~$d$.
  Let $\alpha' := \pi^x \alpha^y$.
  We have $(\alpha')^{u'} = \pi^{xu'} \alpha^{yu'} = \alpha^{xd+yu'} = \alpha$, so $K(\alpha) = K(\alpha')$, and $(\alpha')^d = \pi^{xd}\alpha^{yd}=\pi^{xd}\pi^{yu'} = \pi$, so $\alpha'$ is a root of the polynomial $(\alpha')^d - \pi$, which is Eisenstein.
  Therefore, $K(\alpha) = K(\alpha')$ is a totally ramified $C_d$-extension of~$K$, so its ramification index is indeed~$d$.
\end{proof}

\section{$(C_p \rtimes C_{p-1})$-extensions in characteristic~$p$}

Let~$K$ be a field of characteristic~$p$.

\subsection{Twisted Artin--Schreier theory}
Let $a \in K^\times$, and fix a $(p-1)$-th root~$\alpha \in (K^\sep)^\times$ of~$a$.

\begin{definition}
  The \emph{$a$-twisted Artin--Schreier map} is the map $\wpa \colon K^\sep \to K^\sep$, $x \mapsto x^p - a x$.
\end{definition}

The map~$\wpa$ is surjective and $\F_p$-linear.
The kernel of~$\wpa$ consists of~$0$ together with the $(p-1)$-th roots of~$a$, i.e., $\ker\wpa = \alpha \F_p$.
We have the short exact sequence
\begin{equation}
  \label{eq:twisted-artsch-1}
  0
  \longrightarrow
  \alpha \F_p
  \overset\subset\longrightarrow
  K^\sep
  \overset{\wpa}\longrightarrow
  K^\sep
  \longrightarrow
  0.
\end{equation}

\begin{remark}
  \label{rk:wp-injective}
  The restricted map $(\wpa)|_K \colon K \to K$ has kernel $\alpha \F_p \cap K$.
  Therefore, it is injective when $a \not\in (K^\times)^{p-1}$, and otherwise has kernel~$\alpha \F_p$ of size~$p$.
\end{remark}

In \Cref{eq:twisted-artsch-1}, the term~$\alpha \F_p$ is group-theoretically isomorphic to~$C_p$ (via division by~$\alpha$), but the $\Gamma_K$-action is ``twisted'' by the Kummer homomorphism $\chi_a \colon \tau \mapsto \tau(\alpha)/\alpha$.
By \cite[Chap.~X, §1, Prop.~1]{serrecl}, we have $H^1(\Gamma_K, K^\sep) = 0$, so we obtain the exact sequence in Galois cohomology
\begin{equation}
  \label{eq:twisted-artsch}
  0
  \longrightarrow
  \alpha \F_p \cap K
  \overset\subset\longrightarrow
  K
  \overset{\wpa}\longrightarrow
  K
  \longrightarrow
  H^1(\Gamma_K, \alpha \F_p)
  \longrightarrow
  0,
\end{equation}
whence an isomorphism $H^1(\Gamma_K, \alpha \F_p) \simeq K/\wpa(K)$.
We show in the same way that $H^i(\Gamma_K, \alpha \F_p) = 0$ for all $i \geq 2$.
For any $\lambda, a \in (K^\sep)^\times$, $\gamma \in K^\sep$, we have
\begin{equation}
  \label{eq:change-of-a}
  \wps {\lambda^{p-1}a} (\lambda \gamma)
  =
  \lambda^p \gamma^p - \lambda^{p-1} a \lambda \gamma
  =
  \lambda^p \wpa(\gamma).
\end{equation}
Applying this to $\lambda=\alpha^{-1}$ yields the following equality, relating twisted Artin--Schreier theory to ordinary Artin--Schreier theory:
\begin{equation}
  \label{prop:twist-vs-non-twist}
  \wpa(\gamma)
  = \alpha^p \wp(\alpha^{-1} \gamma)
  = a \alpha \wp(\alpha^{-1}\gamma).
\end{equation}
In particular, we have the equality of fields $K(\alpha, \wpa^{-1}(g)) = K(\alpha, \wp^{-1}(\alpha^{-p} g))$ for any $g \in K$.
Another consequence of \Cref{eq:change-of-a} is that, for any $\lambda \in K^\times$, multiplication by $\lambda^p$ induces an isomorphism $K/\wpa(K) \simto K/\wps {\lambda^{p-1}a}(K)$ of $\F_p$-vector spaces.

\subsection{Parametrization of $(C_p\rtimes C_{p-1})$-extensions}

We use \cite[Proposition~2.4]{wildcount} to parametrize $(C_p \rtimes C_{p-1})$-extensions in characteristic~$p$:%
\footnote{
  The same principle allows one to parametrize $\cG(\F_p)$-extensions for any algebraic group~$\cG$ over~$\F_p$ that can be resolved (using short exact sequences) to factors that are all~$\bG_a$ or~$\bG_m$ (or even~$\GL_n$, cf.~\cite[Chap.~X, §1, Prop.~3]{serrecl}).
  For instance, this applies to $(\F_p^n \rtimes \GL_n(\F_p))$-extensions, i.e., to affine groups of higher dimension.
}

\begin{theorem}
  \label{thm:main-param}
  To any $a \in K^\times$ and any $g \in K$, we can associate a $(C_p \rtimes C_{p-1})$-extension~$[\rho_{g, a}] \in H^1(\Gamma_K, C_p \rtimes C_{p-1})$, which is the $G$-conjugacy class consisting of the following continuous homomorphisms $\Gamma_K \to C_p \rtimes C_{p-1}$:
  \begin{equation}
    \label{def-rhoga}
    \tau
    \longmapsto
    \leftl(
      \frac
        {\tau(\gamma)-\gamma}
        \alpha
      , \,
      \frac
        {\tau(\alpha)}
        \alpha
    \rightr)
    \qqquad
    \text{for }
    \begin{cases}
      \alpha \in (K^\sep)^\times \text{ a $(p-1)$-th root of } a \\
      \gamma \in \wpa^{-1}(g).
    \end{cases}
  \end{equation}
  Moreover:
  \begin{enumroman}
    \item
      \label{thm:main-param-full}
      For any $[\rho] \in H^1(\Gamma_K, C_p \rtimes C_{p-1})$, there exist $g \in K$ and $a \in K^\times$ such that $[\rho] = [\rho_{g, a}]$;
    \item
      \label{thm:main-param-unique}
      For a fixed $a \in K^\times$, we have $[\rho_{g, a}] = [\rho_{g', a}]$ if and only if there exists $\mu \in \F_p^\times$ such that $g' - \mu g \in \wpa(K)$.
    \item
      \label{thm:main-param-change-of-a}
      If $a, \lambda \in K^\times$ and $g \in K$, then $[\rho_{\lambda^p g, \, \lambda^{p-1} a}] = [\rho_{g, a}]$.
  \end{enumroman}
\end{theorem}

Later, whenever we discuss conjugation-invariant properties of homomorphisms, we shall freely refer to~$\rho_{g, a}$ as if it were a well-determined continuous homomorphism $\Gamma_K \rightarrow C_p \rtimes C_{p-1}$.

\begin{proof}
  Let $G := C_p \rtimes C_{p-1} \simeq \F_p \rtimes \F_p^\times$, let $G_K := K \rtimes K^\times$, and let $G_{K^\sep} := K^\sep \rtimes (K^\sep)^\times$ (for the natural action of~$\bG_m$ on~$\bG_a$), all equipped with their natural endomorphisms~$\sigma$ induced by~$x \mapsto x^p$ on both~$\bG_a$ and~$\bG_m$.
  Observe that the fixed points of~$\sigma$ are in each case exactly the elements of~$G$.
  Define an action of~$G_{K^\sep}$ on itself (the \emph{$\sigma$-conjugation action}) as follows, for any $\gamma, g \in K^\sep$ and any $\alpha, a \in (K^\sep)^\times$:
  \begin{align}
    (\gamma, \alpha).(g, a)
    &:=
    \sigma(\gamma, \alpha)
    (g, a)
    (\gamma, \alpha)^{-1}
    \notag
    \\&=
    (\gamma^p + \alpha^p g, \alpha^p a)
    (-\alpha^{-1}\gamma, \alpha^{-1})
    =
    (\gamma^p+ \alpha^p g - \alpha^{p-1} a \gamma , \alpha^{p-1} a)
    \label{sigma-conj}
  \end{align}
  and define the multiplicative Artin--Schreier map $\bwp \colon G_{K^\sep} \to G_{K^\sep}$ by
  \[
    \bwp(\gamma, \alpha)
    :=
    (\gamma, \alpha).(0,1)
    =
    (
      \gamma^p -
      \alpha^{p-1}\gamma
      , \,
      \alpha^{p-1}
    )
    =
    (
      \wps {\alpha^{p-1}}(\gamma)
      , \,
      \alpha^{p-1}
    ).
  \]
  The map~$\bwp \colon G_{K^\sep} \to G_{K^\sep}$ is surjective: we obtain a preimage~$(\gamma,\alpha)$ of any element~$(g, a) \in G_{K^\sep}$ by picking any $(p-1)$-th root~$\alpha$ of~$a$ and any $\gamma \in \wpa^{-1}(g)$.
  Moreover, $H^1(\Gamma_K, G_{K^\sep})$ is the trivial pointed set, as it fits into the exact sequence of pointed sets in non-abelian Galois cohomology
  \[
    H^1(\Gamma_K, K^\sep)
    \longrightarrow
    H^1(\Gamma_K, G_{K^\sep})
    \longrightarrow
    H^1(\Gamma_K, (K^\sep)^\times)
  \]
  in which both the left-hand side and the right-hand side vanish, respectively, by the additive and by the multiplicative versions of Hilbert's Theorem 90, cf.~\cite[Chap.~X, §1, Prop.~1 and~2]{serrecl}.
  By \cite[Proposition~2.4]{wildcount}, there is a bijection between the set of $\sigma$-conjugacy classes of~$G_K$ (orbits under $\sigma$-conjugation) and the pointed set~$H^1(\Gamma_K, G)$ of $G$-conjugacy classes of continuous group homomorphisms $\Gamma_K \to G$.
  Explicitly, for any $(g, a) \in G_K$, a corresponding homomorphism~$\rho \in \Hom(\Gamma_K, G)$ is obtained by picking some $(\gamma, \alpha) \in \bwp^{-1}(g, a)$ and by defining
  \[
    \rho(\tau) :=
    (\gamma, \alpha)^{-1} \tau(\gamma, \alpha)
    =
    (-\alpha^{-1}\gamma, \alpha^{-1})(\tau(\gamma), \tau(\alpha))
    =
    (
      -\alpha^{-1}\gamma + \alpha^{-1}\tau(\gamma), \alpha^{-1}\tau(\alpha)
    )
  \]
  which matches the definition of $[\rho_{g, a}]$ in the statement of the theorem.
  This proves the main claim and point~\ref{thm:main-param-full}.

  We now prove point~\ref{thm:main-param-unique}.
  By the bijection described above, we know that $[\rho_{g, a}] = [\rho_{g', a}]$ if and only if there exist $\delta \in K$, $\mu \in K^\times$ such that $(\delta, \mu).(g, a) = (g', a)$.
  By \Cref{sigma-conj}, this means in particular that $\mu^{p-1} = 1$, i.e., $\mu \in \F_p^\times$, and under that assumption it is equivalent to the equality $\delta^p + \mu g - a \delta = g'$, i.e., $g'-\mu g = \wpa(\delta)$.
  The claim follows.

  Finally, point~\ref{thm:main-param-change-of-a} follows from \Cref{eq:change-of-a} (if $\gamma \in \wpa^{-1}(g)$, then~$\lambda\gamma \in \wps {\lambda^{p-1}a}^{-1}(\lambda^p g)$) and from the following equalities, valid for any $\lambda \in K^\times$:
  \[
    \frac{
      \tau(\lambda\gamma)
      -
      \lambda\gamma
    }{
      \lambda \alpha
    }
    =
    \frac{
      \tau(\gamma)
      -
      \gamma
    }{
      \alpha
    }
    \qquad\qquad\text{and}\qquad\qquad
    \frac{\tau(\lambda\alpha)}{\lambda\alpha}
    =
    \frac{\tau(\alpha)}{\alpha}.
    \qedhere
  \]
\end{proof}

To sum things up: any $(C_p \rtimes C_{p-1})$-extension of~$K$ is induced by some pair~$(g,a)$ with~$g \in K$ and~$a \in K^\times$.
Ignoring surjectivity questions, we summarize this by the field diagram
\[\begin{tikzcd}
	{K(a^{\frac1{p-1}}, \wpa^{-1}(g))} \\
	& {K(a^{\frac1{p-1}})} \\
	K
	\arrow["{C_p}"', no head, from=2-2, to=1-1]
	\arrow["{C_p \rtimes C_{p-1}}", no head, from=3-1, to=1-1]
	\arrow["{C_{p-1}}"', no head, from=3-1, to=2-2]
\end{tikzcd}\]
Recall from \Cref{prop:twist-vs-non-twist} that $K(a^{\frac1{p-1}}, \wpa^{-1}(g)) = K(a^{\frac1{p-1}}, \wp^{-1}(a^{\frac{-p}{p-1}}g))$.


\begin{remark}
  We can also prove \Cref{thm:main-param} by observing that any $\chi_a \in \Hom(\Gamma_K, C_{p-1})$ admits the lift $\rho_{0,a} := s \circ \chi_a \in \Hom(\Gamma_K, C_p \rtimes C_{p-1})$, where $s \colon C_{p-1} \to C_p \rtimes C_{p-1}$ is the canonical (homomorphic) section.
  Then, all lifts of~$\chi_a$ are of the form~$\rho := \tilde \chi \cdot \rho_{0,a}$ (pointwise product) for some continuous map $\tilde\chi \colon \Gamma_K \to \F_p$ satisfying
  the cocycle condition
  $\tilde\chi(\tau_1\tau_2)
  = \tilde\chi(\tau_1) + \chi_a(\tau_1)\tilde\chi(\tau_2)$: in other words, the abelian group of $1$-cocycles (continuous crossed homomorphisms) $\alpha\tilde\chi \in Z^1(\Gamma_K, \alpha \F_p)$ acts simply transitively on the set of lifts of~$\chi_a$.
  Combined with \Cref{eq:kummer,eq:twisted-artsch}, we retrieve the parametrization of \Cref{thm:main-param}.
\end{remark}

\subsection{Fiberwise parametrization in terms of homomorphisms.}
\label{subsn:fiberwise}

Consider the canonical surjection $\pi \colon C_p \rtimes C_{p-1} \twoheadrightarrow C_{p-1}$.
Fix~$a \in K^\times$, and pick a $(p-1)$-th root~$\alpha \in (K^\sep)^\times$ of~$a$.
Then, \Cref{thm:main-param} becomes a bijection
\begin{equation}
  \label{eq:fiberwise-param}
  \begin{matrix}
    \wpa^{-1}(K)/K
    &\overset\sim\longrightarrow&
    \suchthat{
      \rho \in \Hom(\Gamma_K, C_p \rtimes C_{p-1})
    }{
      \pi\circ\rho = \chi_a
    }
    \\
    [\gamma]
    &
    \longmapsto
    &
    \displaystyle
    \tau \mapsto
    \Bigl(
      \frac
        {\tau(\gamma)-\gamma}
        \alpha
      ,\,
      \chi_a(\tau)
    \Bigr).
  \end{matrix}
\end{equation}
Note that the $\F_p$-vector space $\wpa^{-1}(K)/K$ also appears in the short exact sequence
\begin{equation}
  \label{eq:fiberwise-2}
  0
  \longrightarrow
  \alpha \F_p/(K\cap \alpha \F_p)
  \overset\subset\longrightarrow
  \wpa^{-1}(K)/K
  \overset{\wpa}\longrightarrow
  K/\wpa(K)
  \longrightarrow
  0
\end{equation}
in which $\alpha \F_p/(K\cap \alpha \F_p)$ is trivial if $a \in (K^\times)^{p-1}$, and is~$\alpha \F_p$ (of size~$p$) otherwise.
Consequently, when we count homomorphisms $\rho \in \Hom(\Gamma_K, C_p \rtimes C_{p-1})$ satisfying certain conditions by counting the corresponding elements of~$K/\wpa(K)$, we have to multiply the counts by~$p$ whenever~$a \not\in(K^\times)^{p-1}$.

\section{Twisted Artin--Schreier theory over local function fields}
\label{sn:local}

In this section, we fix a local function field~$K = \kappa\ppar\pi$ with perfect residue field~$\kappa$ of characteristic~$p$ and uniformizer~$\pi$, with ring of integers $\cO = \kappa\bbra\pi$, with maximal ideal $\fm = \pi\cO$, equipped with its (discrete) $\pi$-adic valuation~$v$.
The maximal unramified extension of~$K$ is $K^\ur := \bar\kappa\ppar\pi$.

Fix $a \in K^\times$ and let $\alpha \in (K^\sep)^\times$ be a $(p-1)$-th root of~$a$.
Let $u := v(a)$, and let $A \in \kappa$ be the coefficient of~$a$ in front of~$\pi^u$.
If~$\kappa$ is a finite field, \Cref{rk:wp-injective} implies that
\begin{equation}
  \label{eq:size-wp-finite}
  \card{\kappa/\wps A (\kappa)}
  =
  \frac{\card\kappa}{\card{\wps A (\kappa)}}
  =
  \card{\ker\wps A \cap \kappa}
  =
  \begin{cases}
    1 & \text{if } A \not\in (\kappa^\times)^{p-1} \\
    p & \text{if } A \in (\kappa^\times)^{p-1}.
  \end{cases}
\end{equation}

We define the integers $k := \floor{\frac u{p-1}}$ and $k' := \ceil{\frac u{p-1}}$.
For convenience, we write down how~$\wpa$ acts on valuations: for any $g \in K$, we have
\begin{equation}
  \label{eq:wp-val}
  v(\wpa(g))
  \begin{cases}
    = p v(g) & \text{if } v(g) < \frac u{p-1} \iff v(g) \leq k'-1\\
    = u+v(g) & \text{if } v(g) > \frac{u}{p-1} \iff v(g) \geq k + 1\\
    \geq \frac{pu}{p-1} & \text{if } v(g) = \frac{u}{p-1} \iff p-1 \mid u \text{ and } v(g)=k=k'.
  \end{cases}
\end{equation}
In the first case, we have $p \mid v(\wpa(g)) \leq p(k'-1) < u+k'-1 < \frac{pu}{p-1}$, and in the second case $v(\wpa(g)) \geq u+k + 1 > \frac{pu}{p-1}$.

\subsection{The image of~$\wpa$.}

We begin with the following valuative criterion for being in the image of~$\wpa$:

\begin{lemma}
  \label{lem:image-valuation-criterion}
  Any $g \in K$ with $v(g) > \frac{pu}{p-1}$ is in $\wpa(K)$.
\end{lemma}

\begin{proof}
  Since~$K$ is ultrametric complete, the following sum converges in~$K$:
  \[
    w :=
    \sum_{i \geq 0}
      a^{-\frac{p^{i+1}-1}{p-1}}
      (-g)^{p^i}.
  \]
  Indeed, the $i$-th summand has valuation $p^i v(g) - \frac{p^{i+1} - 1}{p-1} u = \Bigl( v(g) - \frac{pu}{p-1} \Bigr) p^i + \frac{u}{p-1}$, which tends to~$+\infty$ when $i \to +\infty$ because $v(g)-\frac{pu}{p-1} > 0$.
  We have
  \begin{align*}
    \wpa(w)
    &=
    \sum_{i \geq 0}
      a^{-\frac{p^{i+2}-p}{p-1}}
      (-g)^{p^{i+1}}
    -
    a
    \sum_{i \geq 0}
      a^{-\frac{p^{i+1}-1}{p-1}}
      (-g)^{p^i}
    \\
    &=
    \sum_{i \geq 1}
      a^{-\frac{p^{i+1}-p}{p-1}}
      (-g)^{p^i}
    -
    \sum_{i \geq 0}
      a^{-\frac{p^{i+1}-1}{p-1} + 1}
      (-g)^{p^i}
    =
    g.
    \qedhere
  \end{align*}
\end{proof}

The following lemma extends \Cref{lem:image-valuation-criterion} slightly when $p-1 \mid u$:

\begin{lemma}
  \label{lem:image-finite-criterion}
  Assume that $p-1 \mid u$, so $u = k(p-1)$.
  Then, an element~$g \in K$ with $v(g) = pk$ is in~$\wpa(K)$ if and only if the coefficient~$g_{pk}$ of~$g$ in front of~$\pi^{pk}$ belongs to~$\wps A (\kappa)$.
\end{lemma}

In particular, if~$\kappa$ is algebraically closed or if~$\kappa$ is finite and $A \not\in (\kappa^\times)^{p-1}$ (cf.~\Cref{eq:size-wp-finite}), so that $\wps A (\kappa) = \kappa$, then any $g \in K$ with $v(g) \geq pk$ is in~$\wpa(K)$.

\begin{proof}
  If $g=\wpa(y)$ for some $y \in K$, then by \Cref{eq:wp-val}, we must have $v(y) = k$.
  In terms of $\bar a := \pi^{-u} a$, $\bar g := \pi^{-pk} g$ and $\bar y := \pi^{-k} y$, which all belong to~$\kappa^\times + \pi\cO$, the equation $g=\wpa(y)$ becomes $\bar g = \bar y^p - \bar a \, \bar y$ (we have $k+u=k+(p-1)k=pk$).
  By Hensel's lemma, this separable equation in the variable~$\bar y \in \cO$ has a solution if and only if it has a solution modulo~$\pi\cO$, i.e., if and only if the coefficient $\bar g_0 = g_{pk} \in \kappa^\times$ is in~$\wpa(\kappa)$.
\end{proof}

\subsection{Local fundamental domain}
Assume that~$\kappa$ is finite.

\begin{definition}
  \label{def:local-fundom}
  Define the following $\kappa$-linear subspaces of~$K = \kappa\ppar\pi$:
  \[
    \cF_a' :=
    \bigoplus_{
        p \nmid i < \frac{pu}{p-1}
    }
        \kappa \pi^i
    \andd
    \cF_a :=
    \begin{cases}
      \cF_a'
      &
      \text{if }
      a \not\in (K^\times)^{p-1}
      \\
      \kappa \pi^{pk}
      \oplus
      \cF_a'
      &
      \text{if }
      a \in (K^\times)^{p-1}, \,
      k = \frac{u}{p-1}.
    \end{cases}
  \]
\end{definition}

\begin{theorem}
  \label{thm:local-fundom}
  ~
  \begin{enumroman}
    \item
      \label{thm:local-fundom:full}
      Every $\wpa(K)$-coset in~$K$ intersects~$\cF_a$
    \item
      \label{thm:local-fundom:unique}
      We have
      \[
        \wpa(K) \cap \cF_a =
        \begin{cases}
          0 &
          \text{if }
          a \not\in (K^\times)^{p-1}
          \\
          \wps A (\kappa) \pi^{pk}
          &
          \text{if }
          a \in (K^\times)^{p-1}
        \end{cases}
        \andd
        \wpa(K)
        \cap
        \cF'_a
        =
        0.
      \]
    \item
      \label{thm:local-fundom:is-fundom}
      The map $\cF_a \overset\subseteq\to K \twoheadrightarrow K/\wpa(K)$ is an $\F_p$-linear isomorphism if $a \not\in (K^\times)^{p-1}$, and otherwise is a surjection with kernel $\wps A (\kappa) \pi^{pk}$.
      This can be summarized as
      \[
        \bigslant K {\wpa(K)}
        \simeq
        \begin{cases}
          \cF_a \simeq \cF_a'
          &
          \text{if }
          a \not\in (K^\times)^{p-1}
          \\
          \Bigl(
            \bigslant\kappa{\wps A (\kappa)}
          \Bigr)
          \pi^{pk}
          \oplus \cF_a'
          &
          \text{if }
          a \in (K^\times)^{p-1}.
        \end{cases}
      \]
  \end{enumroman}  
\end{theorem}

\begin{proof}
  To prove~\ref{thm:local-fundom:full}, we prove by decreasing induction on the valuation~$v(g)$ that any $g \in K$ lies in $\cF_a + \wpa(K)$.
  First, by \Cref{lem:image-valuation-criterion}, if $v(g) > \frac{pu}{p-1}$, then $g \in \wpa(K) \subseteq \cF_a + \wpa(K)$.
  Now, assume that $v(g) \leq \frac{pu}{p-1}$ and that any~$h$ with $v(h) > v(g)$ lies in $\cF_a + \wpa(K)$, and write $g = g_{v(g)} \pi^{v(g)} + \tilde g$ where $g_{v(g)} \in \kappa^\times$ and $v(\tilde g) > v(g)$.
  If $v(g) = \frac{pu}{p-1}$ and $a \not\in (K^\times)^{p-1}$, then by \Cref{lem:image-finite-criterion} we have $g \in \wpa(K) \subseteq \cF_a + \wpa(K)$.
  If $p \nmid v(g)$, or if $v(g) = \frac{pu}{p-1}$ and $a \in (K^\times)^{p-1}$, then $g_{v(g)} \pi^{v(g)} \in \cF_a$, and $\tilde g \in \cF_a + \wpa(K)$ by the induction hypothesis, so $g \in \cF_a + \wpa(K)$.
  Assume now that $p \mid v(g) < \frac{pu}{p-1}$, and write $v(g) = pi$ with $i < \frac{u}{p-1}$.
  We have
  \[
    g_{v(g)} \pi^{v(g)}
    =
    \bigl(
      \sigma^{-1}(g_{v(g)}) \pi^i
    \bigr)^p
    =
    \wpa\bigl(
      \sigma^{-1}(g_{v(g)}) \pi^i
    \bigr)
    +
    a \sigma^{-1}(g_{v(g)}) \pi^i
  \]
  and
  $
    v\bigl(
      a \sigma^{-1}(g_{v(g)}) \pi^i
    \bigr)
    =
    u + i
    >
    (p-1)i+i
    =
    pi
    =v(g)
  $, so
  \[
    g
    =
    \underbrace{
      \wpa\bigl(
        \sigma^{-1}(g_{v(g)}) \pi^i
      \bigr)
    }_{\in \wpa(K)}
    +
    \underbrace{
      a \sigma^{-1}(g_{v(g)}) \pi^i
      +
      \tilde g
    }_{\in \cF_a + \wpa(K) \text{ by the induction hypothesis}}
    \in \cF_a + \wpa(K).
  \]

  To prove~\ref{thm:local-fundom:unique}, assume that $x \in K^\times$ is such that $\wpa(x) \in \cF_a \setminus \{0\}$.
  If $v(x) < \frac{u}{p-1}$, then by \Cref{eq:wp-val} $v(\wpa(x)) = pv(x) < \frac{pu}{p-1}$ is divisible by~$p$, which is impossible for elements of~$\cF_a$.
  If $v(x) > \frac{u}{p-1}$, then $v(\wpa(x)) = u + v(x) > \frac{pu}{p-1}$ which is again impossible for non-zero elements of~$\cF_a$.
  Therefore, $v(x) = \frac{u}{p-1}$, so $p-1 \mid u$ and $v(x) = k$.
  The inequality $v(\wpa(x)) \geq pk$ with $\wpa(x) \in \cF_a \setminus \{0\}$ implies that $\wpa(x) \in \kappa^\times \pi^{pk}$, which is only possible if~$a \in (K^\times)^{p-1}$, in which case $\kappa^\times \pi^{pk} \subseteq \cF_a$.
  Finally, by \Cref{lem:image-finite-criterion}, we have $\kappa \pi^{pk} \cap \wpa(K) = \wps A (\kappa) \pi^{pk}$.
  That $\wpa(K) \cap \cF'_a = 0$ easily follows.

  Finally, \ref{thm:local-fundom:is-fundom} is a direct consequence of \ref{thm:local-fundom:full} and \ref{thm:local-fundom:unique}.
\end{proof}

\subsection{Computation of discriminants}

By \Cref{prop:kummer-inertia}, the extension~$K(\alpha)|K$ has ramification index
\begin{equation}
  \label{eq:def-e}
  e := \frac{p-1}{\gcd(p-1, \, u)}.
\end{equation}
Therefore, $v_{K(\alpha)} = e v_K$, and
\begin{equation}
  \label{eq:local-kummer-disc}
  \delta(\chi_a)
  =
  (p-1)\biggl(
    1-\frac1e
  \biggr)
  =
  p-1-
  \gcd\bigl(
    p-1 , \, u
  \bigr)
  \qquad
  \in\{0,\ldots,p-2\}.
\end{equation}

\begin{definition}
  \label{def:local-reldisc-tasty}
  For any $g \in K$, we define
  $
    \delta^{(a)}(g)
    :=
    \delta^\rel(\rho_{g,a})
  $
  (cf.~\Cref{def:disc-rel,thm:main-param}),
  so that $\delta(\rho_{g,a}) = p\delta(\chi_a) + (p-1)\delta^{(a)}(g)$.
\end{definition}

\begin{lemma}
  \label{lem:disc-lj}
  For any $g \in K$, we have
  \[
    \delta^{(a)}(g)
    =
    \begin{cases}
      0
      & \text{if }
      \lastjump \rho_{g,a} = 0 \\
      (p-1)
      \bigl(
        \lastjump \rho_{g,a}
        + \frac1e
      \bigr)
      =
      (p-1)
      \bigl(
        \lastjump \rho_{g,a}
        + 1
      \bigr)
      - \delta(\chi_a)
      & \text{if }
      \lastjump \rho_{g,a} > 0
    \end{cases}.
  \]
\end{lemma}

\begin{proof}
  By definition:
  \begin{align*}
    \delta^{(a)}(g) & =
    p
    \int_{-1}^{+\infty}
      \leftl(
        1 - \frac1{\card{\rho_{g,a}(\Gamma_K^v)}}
      \rightr)
      \dd v
      -
      p
      \Bigl(
        1-\frac1e
      \Bigr).
  \end{align*}
  If $\lastjump \rho_{g,a} = 0$, then~$\rho_{g,a}$ is tamely ramified, so $\card{\rho_{g,a}(\Gamma_K^v)}$ equals~$e$ for $v \in (-1,0]$ and~$1$ for~$v > 0$, and then $\delta^{(a)}(g) = 0$.
  Assume now that $\lastjump \rho_{g,a} > 0$, so
  \[
    \card{\rho_{g,a}(\Gamma_K^v)}
    =
    \begin{cases}
      ep & \text{if } -1 < v \leq 0 \\
      p & \text{if } 0 < v \leq \lastjump \rho_{g,a} \\
      1 & \text{if } v > \lastjump \rho_{g,a}
    \end{cases}
  \]
  whence
  \begin{align*}
    \delta^{(a)}(g)
    &=
    p
    \left(
      1
      \cdot
      \Bigl(
        1-\frac1{ep}
      \Bigr)
      +
      \lastjump \rho_{g,a}
      \cdot
      \Bigl(
        1-\frac1p
      \Bigr)
      -
      \Bigl(
        1-\frac1e
      \Bigr)
    \right)
    \\
    &=
    p
    \left(
      \lastjump \rho_{g,a}
      +
      \frac1e
    \right)
    \left(
      1-\frac1p
    \right)
    =
    (p-1)
    \left(
      \lastjump \rho_{g,a}
      +
      \frac1e
    \right).
  \end{align*}
  The remaining equality comes from the fact that $\frac{p-1}e = p-1-\delta(\chi_a)$, cf.~\Cref{eq:local-kummer-disc}.
\end{proof}

We will need the following classical fact about the ramification filtration in the upper numbering, for which we did not find a direct reference:

\begin{lemma}
  \label{lem:subgroup-upnum}
  For any $v > 0$, we have $\Gamma_K^v = \Gamma_{K(\alpha)}^{ev}$.
\end{lemma}

\begin{proof}
  It suffices to prove that for any finite Galois extension~$L|K$ containing~$K(\alpha)$ we have $\Gal(L|K)^v = \Gal(L|K(\alpha))^{ev}$.
  Since~$K(\alpha)|K$ is at most tamely ramified and $v > 0$, we know that $\Gal(L|K)^v \subseteq \Gal(L|K(\alpha))$, so
  \begin{align*}
    \Gal(L|K)^v
    & =
    \Gal(L|K)^v \cap \Gal(L|K(\alpha))
    \tag*{as $\Gal(L|K)^v \subseteq \Gal(L|K(\alpha))$}
    \\
    & =
    \Gal(L|K)_{\psi_{L|K}(v)} \cap \Gal(L|K(\alpha))
    \tag*{by definition, cf.~\cite[Chap.~IV, §3]{serrecl}}
    \\
    & =
    \Gal(L|K(\alpha))_{\psi_{L|K}(v)}
    \tag*{\cite[Chap.~IV, §1, Prop.~2]{serrecl}}
    \\
    & =
    \Gal(L|K(\alpha))_{\psi_{L|K(\alpha)}(\psi_{K(\alpha)|K}(v))}
    \tag*{\cite[Chap.~IV, §3, Prop.~15]{serrecl}}
    \\
    & =
    \Gal(L|K(\alpha))^{\psi_{K(\alpha)|K}(v)}
    \tag*{by definition.}
  \end{align*}
  The size of~$\chi_a(\Gamma_K^s)$ is~$e$ if~$s = 0$, and it is~$1$ if~$s > 0$ because~$K(\alpha)|K$ is tamely ramified.
  Hence:
  \[
    \psi_{K(\alpha)|K}(v)
    =
    \int_0^v
      \frac{
        \card{\chi_a(\Gamma_K^0)}
      }{
        \card{\chi_a(\Gamma_K^s)}
      }
      \dd s
    =
    \int_0^v
      e
      \dd s
    =
    ev.
    \qedhere
  \]
\end{proof}

\begin{theorem}[Discriminants]
  \label{thm:local-disc}
  Let $g \in K$.
  \begin{enumroman}
    \item
      \label{thm:local-disc:ultrametric}
      For any $h \in K$, we have the ``ultrametric inequality''
      \[
        \delta^{(a)}(g+h)
        \leq
        \max\Bigl(
          \delta^{(a)}(g), \,
          \delta^{(a)}(h)
        \Bigr)
      \]
      with equality when $\delta^{(a)}(g) \neq \delta^{(a)}(h)$, and similarly for $\delta(\rho_{g,a})$.
    \item
      \label{thm:local-disc:unramified-twist}
      If $h \in K$ is such that $\delta^{(a)}(h) = 0$, then $\delta^{(a)}(g) = \delta^{(a)}(g+h)$.
    \item
      \label{thm:local-disc:unramified-criterion}
      If $v(g) \geq \frac{pu}{p-1}$, then $\delta^{(a)}(g) = 0$.
    \item
      \label{thm:local-disc:unramified-converse}
      If $\delta^{(a)}(g) = 0$ and $p \nmid v(g)$, then $v(g) > \frac{pu}{p-1}$.
    \item
      \label{thm:local-disc:formula}
      We have the inequality
      \[
        \delta^{(a)}(g)
        \leq
        \max\Bigl(0, \,
          p u
          +
          (p-1)
          \bigl(
            1-v(g)
          \bigr)
          - \delta(\chi_a)
        \Bigr)
      \]
      with equality if and only if $p \nmid v(g) < \frac{pu}{p-1}$, or $v(g) > \frac{pu}{p-1}$.
    \item
      \label{thm:local-disc:formula-bis}
      If $p \nmid v(g)$ or if $v(g) \geq \frac{pu}{p-1}$ (e.g., if $g \in \cF_a$), then
      \begin{align*}
        \delta^{(a)}(g) & =
        \leftl\lbrace
        \begin{array}{l}
          0 \\
          pu
          +
          (p-1)
          \bigl(
            1-v(g)
          \bigr)
          - \delta(\chi_a)
        \end{array}
        \rightr.
        &
        \begin{array}{l}
          \text{if } v(g) \geq \frac{pu}{p-1} \\
          \text{if } p \nmid v(g) < \frac{pu}{p-1}
        \end{array}
        \\
        \delta(\rho_{g,a}) & =
        \leftl\lbrace
        \begin{array}{l}
          p\delta(\chi_a) \\
          (p-1)pu
          +
          (p-1)^2
          \bigl(
            1-v(g)
          \bigr)
          + \delta(\chi_a)
        \end{array}
        \rightr.
        &
        \begin{array}{l}
          \text{if } v(g) \geq \frac{pu}{p-1} \\
          \text{if } p \nmid v(g) < \frac{pu}{p-1}.
        \end{array}
      \end{align*}
  \end{enumroman}
\end{theorem}


\begin{proof}
  To prove~\ref{thm:local-disc:ultrametric}, by \Cref{lem:disc-lj}, it suffices to prove the ultrametric inequality for the last jump.
  Fix elements $\gamma \in \wpa^{-1}(g)$ and $\gamma' \in \wpa^{-1}(h)$.
  We have $\wpa(\gamma+\gamma') = g+h$, so for any $v > 0$ we have (cf.~\Cref{def-rhoga}, and recall that~$\chi_a$ is at most tamely ramified)
  \[
    \lastjump \rho_{g+h, a} < v
    \iff
    \forall \tau \in \Gamma_K^v, \;
    \tau(\gamma+\gamma') = \gamma+\gamma'.
  \]
  If $\lastjump \rho_{g, a} < v$ and $\lastjump \rho_{h, a} < v$, then for all $\tau \in \Gamma_K^v$ we have $\tau(\gamma)=\gamma$ and $\tau(\gamma') = \gamma'$, so $\tau(\gamma+\gamma') = \gamma+\gamma'$ and $\lastjump \rho_{g+h, a} < v$, proving the inequality.
  Now, assume that $\lastjump \rho_{g, a} \neq \lastjump \rho_{h, a}$, for example that $\lastjump \rho_{g, a} < \lastjump \rho_{h, a}$, and assume by contradiction that $\lastjump \rho_{g+h,a} < \lastjump \rho_{h, a}$.
  By \iref{thm:main-param}{thm:main-param-unique}, we have $\rho_{-g,a} = \rho_{g, a}$, so using the ultrametric inequality we obtain the contradictory inequality
  \begin{align*}
    \lastjump \rho_{h,a}
    & =
    \lastjump \rho_{(g+h)+(-g), a}
    \\
    &
    \leq
    \max\Bigl(
      \lastjump \rho_{g+h, a}, \,
      \lastjump \rho_{g, a}
    \Bigr)
    <
    \lastjump \rho_{h,a}.
  \end{align*}
  We have proved~\ref{thm:local-disc:ultrametric}.
  Point~\ref{thm:local-disc:unramified-twist} follows from~\ref{thm:local-disc:ultrametric} and from the bounds $\delta \geq 0$, $\delta^{(a)} \geq 0$.

  Fix $\gamma \in \wpa^{-1}(g)$, and define as in \Cref{def-rhoga} the continuous group homomorphism
  \[
    \begin{matrix}
      \rho_{g,a} \colon &
      \Gamma_K &
      \longrightarrow &
      C_p \rtimes C_{p-1}
      \\ &
      \tau &
      \longmapsto &
      \leftl(
        \frac{\tau(\gamma)-\gamma}\alpha
        ,\;
        \chi_a(\tau)
      \rightr).
    \end{matrix}
  \]
  By \Cref{lem:subgroup-upnum}, for any $v > 0$, we have $\Gamma_K^v = \Gamma_{K(\alpha)}^{ev}$.
  Therefore:
  \[
    \lastjump \rho_{g,a}
    =
    \frac1e
    \lastjump\bigl(
      \rho_{g,a}|_{\Gamma_{K(\alpha)}}
    \bigr).
  \]
  Note that $\Gamma_{K(\alpha)} = \ker \chi_a$.
  We define the map $\tilde\rho \colon \Gamma_{K(\alpha)} \to C_p$, $\tau \mapsto \frac{\tau(\gamma) - \gamma}\alpha = \tau(\frac\gamma\alpha) - \frac\gamma\alpha$, which is the Artin--Schreier character in $\Hom(\Gamma_{K(\alpha)}, C_p)$ associated to the element
  \[
    \wp\leftl(
      \frac\gamma\alpha
    \rightr)
    =
    \frac{\gamma^p}{\alpha^p} - \frac\gamma\alpha
    =
    \frac{
      (a\gamma + g) -
      \gamma\alpha^{p-1}
    }{\alpha^p}
    =
    \alpha^{-p} g \in K(\alpha).
  \]
  For any $\tau \in \Gamma_{K(\alpha)}$, we have $\rho_{g,a}(\tau) = 1$ if and only if $\tilde\rho(\tau) = 1$ as~$\chi_a$ is tamely ramified.
  Therefore, $\lastjump (\rho_{g,a}|_{\Gamma_{K(\alpha)}}) = \lastjump \tilde\rho$.
  Combining these equalities with \Cref{lem:disc-lj}, we obtain
  \begin{equation}
    \label{eq:reldisc-lj}
    \delta^{(a)}(g)
    =
    \begin{cases}
      0
      &
      \text{if }
      \lastjump \tilde\rho = 0
      \\
      \frac{p-1}e
      (\lastjump\tilde\rho + 1)
      &
      \text{if }
      \lastjump \tilde\rho > 0
    \end{cases}
  \end{equation}
  and in particular $\delta^{(a)}(g) \in \Z_{\geq 0}$.
  It remains to compute the last jump of~$\tilde\rho$.
  In $K(\alpha)$, the element~$\alpha^{-p}g$ has valuation
  \[
    v_{K(\alpha)}(\alpha^{-p}g)
    =
    -p \frac{v_{K(\alpha)}(a)}{p-1} + v_{K(\alpha)}(g)
    =
    \frac{-peu}{p-1} + e v(g).
  \]
  Note that $\frac{eu}{p-1}$ is an integer (namely~$\frac{u}{\gcd(u,\,p-1)}$) so $p | \frac{-peu}{p-1}$.
  Moreover, $e|p-1$ so~$p \nmid e$.
  Therefore, $p$ divides $v_{K(\alpha)}(\alpha^{-p}g)$ if and only if~$p$ divides~$v(g)$.
  Note also that $v_{K(\alpha)}(\alpha^{-p}g) < 0$ if and only if~$v(g) < \frac{pu}{p-1}$.
  The last jump of Artin--Schreier characters is computed in \cite[p.~42]{hasse} or \cite[Proposition~2.1]{larathomas}: we have
  \begin{equation}
    \label{eq:lj-tilderho}
    \lastjump\tilde\rho
    \leq
    \max\bigl(0,
      -v_{K(\alpha)}(\alpha^{-p}g)
    \bigr)
    =
    \max\leftl(0,\,
      \frac{peu}{p-1} - e v(g)
    \rightr)
  \end{equation}
  with equality if and only if $v_{K(\alpha)}(\alpha^{-p}g) \geq 0$ or $p \nmid v_{K(\alpha)}(\alpha^{-p}g)$, i.e., if and only if~$v(g) \geq \frac{pu}{p-1}$ or~$p \nmid v(g)$.

  If $v(g) \geq \frac{pu}{p-1}$ then $\lastjump \tilde\rho = 0$ by \Cref{eq:lj-tilderho} so~$\delta^{(a)}(g) = 0$ by \Cref{eq:reldisc-lj}, showing~\ref{thm:local-disc:unramified-criterion}.
  If~$p \nmid v(g)$, then we have the equivalences (cf.~\Cref{eq:reldisc-lj,eq:lj-tilderho})
  \[
    \delta^{(a)}(g) = 0 \iff \lastjump \tilde\rho = 0 \iff v(g) \geq \frac{pu}{p-1},
  \]
  proving \ref{thm:local-disc:unramified-converse}.
  (We cannot have~$v(g) = \frac{pu}{p-1}$ as otherwise~$p-1\mid u$ and then~$p \mid v(g)$.)

  We now prove~\ref{thm:local-disc:formula}.

  Assume first that $\delta^{(a)}(g) > 0$, which by \Cref{eq:reldisc-lj} amounts to $\lastjump \tilde\rho \geq 1$.
  We have
  \[
    \delta^{(a)}(g) =
    \frac{p-1}e (\lastjump \tilde \rho + 1)
    \leq
    \frac{p-1}e +
    \max\leftl(
      0,\,
      pu - (p-1) v(g)
    \rightr)
  \]
  Recall that $\frac{p-1}e = \gcd(p-1, \, u) = (p-1)-\delta(\chi_a)$.
  By~\ref{thm:local-disc:unramified-criterion}, since $\delta^{(a)}(g) > 0$, we have $v(g) < \frac{pu}{p-1}$, so $\max\leftl(
        0,\,
        pu - (p-1) v(g)
      \rightr) = pu-(p-1)v(g)$ and then
  \[
    \delta^{(a)}(g) \leq
    \underbrace{
      pu+(p-1)(1-v(g))-\delta(\chi_a)
    }_{\geq 0}
  \]
  with equality if and only if we have equality in \Cref{eq:lj-tilderho}.
  This shows \ref{thm:local-disc:formula} in this case.
  (The case $v(g) = \frac{pu}{p-1}$ never occurs in this case by~\ref{thm:local-disc:unramified-criterion}.)

  Now, assume that $\delta^{(a)}(g) = 0$, which by \Cref{eq:reldisc-lj} is equivalent to $\lastjump \tilde\rho = 0$.
  The inequality is obvious in this case, so we focus on the equality case.
  We have equality if and only if $pu+(p-1)(1-v(g))-\delta(\chi_a) \leq 0$, if and only if $v(g) \geq 1 + \frac{pu}{p-1} - \frac{\delta(\chi_a)}{p-1} = \frac{pu}{p-1} + \frac{\gcd(u, p-1)}{p-1}$.
  If that is the case, we clearly have $v(g) > \frac{pu}{p-1}$.
  For the converse, we cannot have $p \nmid v(g) < \frac{pu}{p-1}$ by \ref{thm:local-disc:unramified-converse}, so assume that $v(g) > \frac{pu}{p-1}$, i.e., that $v(g) \geq u+k+1$.
  If $p-1 \mid u$, then $u=(p-1)k$ and $\delta(\chi_a) = 0$, so $v(g) \geq 1 + \frac{pu}{p-1} = 1 + \frac{pu}{p-1} - \frac{\delta(\chi_a)}{p-1}$ and we are done.
  If $p-1 \nmid u$, we have
  \begin{align*}
    pu+(p-1)(1-v(g))-\delta(\chi_a)
    & \leq
    pu-(p-1)(u+k)-\delta(\chi_a)
    \\ & =
    u-(p-1)(k+1)+\gcd(u, p-1)
    \\ & =
    \underbrace{u-(p-1)(k+1)}_{< 0}+\underbrace{\gcd\bigl(u-(p-1)(k+1), \, p-1\bigr)}_{\leq \Bigl| u-(p-1)(k+1) \Bigr|} \leq 0.
  \end{align*}

  Finally, \ref{thm:local-disc:formula-bis} follows directly from~\ref{thm:local-disc:unramified-criterion}, from~\ref{thm:local-disc:formula}, and from \Cref{def:local-reldisc-tasty}.
\end{proof}

\begin{proposition}[Unramified elements]
  \label{prop:local-unram}
  Let $g \in K \setminus \wpa(K)$.
  The following are equivalent:
  \begin{enumroman}
    \item
      $\delta(\rho_{g,a}) = 0$
    \item
      $\delta^{(a)}(g) = 0$
    \item
      $a \in (K^\times)^{p-1}$ and $g \in \wpa(K) + \kappa^\times \pi^{pk}$.
  \end{enumroman}
\end{proposition}

\begin{proof}
  Adding an element of~$\wpa(K)$ if needed, we can assume that $g \in \cF_a \setminus \{0\}$, cf.~\iref{thm:main-param}{thm:main-param-unique} and \iref{thm:local-fundom}{thm:local-fundom:full}.
  In particular, we have $v(g) \leq \frac{pu}{p-1}$.
  
  Assume first that $\delta(\rho_{g,a}) = 0$.
  Then, $\delta(\chi_a) = - \frac{p-1}p \delta^{(a)}(g)$ is divisible by~$p-1$, so $p-1 \mid u$, $e=1$, $\delta(\chi_a)=0$, and then $\delta^{(a)}(g) = 0$.

  Assume now that $\delta^{(a)}(g) = 0$.
  By \iref{thm:local-disc}{thm:local-disc:unramified-converse}, this is impossible if $p \nmid v(g)$, so $p \mid v(g)$ and then $v(g) = \frac{pu}{p-1}$ as $g \in \cF_a \setminus \{0\}$.
  In particular, $p-1 \mid u$, so $u=(p-1)k$ and $v(g) = pk$.
  Since $g \in \cF_a$, this is only possible if $a \in (K^\times)^{p-1}$, and then $g \in \kappa^\times \pi^{pk} \subseteq \wpa(K) + \kappa^\times \pi^{pk}$.

  Finally, if $a \in (K^\times)^{p-1}$ and $v(g) = pk$, then $\delta(\chi_a)=0$ and by \iref{thm:local-disc}{thm:local-disc:unramified-criterion} we have $\delta^{(a)}(g) = 0$, so $\delta(\rho_{g,a}) = p\delta(\chi_a)+(p-1)\delta^{(a)}(g) = 0$.
\end{proof}

Consequently, the $\F_p$-linear subspace (stable under addition by \iref{thm:local-disc}{thm:local-disc:unramified-twist})
\[
  \cF_a^\ur := \suchthat{g \in \cF_a}{\delta^{(a)}(g)=0}
\]
admits the description
\[
  \cF_a^\ur =
  \begin{cases}
    0 & \text{if } a \not\in (K^\times)^{p-1} \\
    \kappa\pi^{pk} & \text{if } a \in (K^\times)^{p-1}
  \end{cases}
\]
A consequence is that the inverse of the inclusion map induces an isomorphism
\begin{equation}
  \label{eq:local-unram-F'}
  \cF_a/\cF_a^\ur
  \overset\sim\longrightarrow
  \cF'_a = \bigoplus_{p \nmid i < \frac{pu}{p-1}} \kappa \pi^i.
\end{equation}
By \iref{thm:local-disc}{thm:local-disc:unramified-twist}, $\delta^{(a)}$ is a well-defined map on the quotient space~$\cF_a/\cF_a^\ur \simeq \cF'_a$ (for elements of~$\cF_a$, it depends only on their projection to~$\cF'_a$).


\section{Counting $(C_p \rtimes C_{p-1})$-extensions of local function fields}
\label{sn:local-count}

The setup is similar to \Cref{sn:local}: we fix a local field $K = \kappa\ppar\pi$ and use the same notations.
We also assume that the residue field~$\kappa$ is finite, and we do not fix an element~$a \in K^\times$ beforehand.
We let $\pi \colon C_p \rtimes C_{p-1} \twoheadrightarrow C_{p-1}$ be the natural surjection.

\subsection{Fiberwise counts}

\begin{definition}
  \label{def:eta-Nux}
  For any $u \in \Z$ and any $x \in \R_{\geq0}$, define
  \[
    \eta^\rel_u(x)
    :=
    \frac{x-\gcd(p-1, \, u)}{p-1}
    - \frac {pu}{p-1}
    \andd
    N^\rel_u(x) :=
    \cardsuchthat{
      i \in \Z
    }{
      p \nmid i \\
      -\eta^\rel_u(x) \leq i < \frac{p u}{p-1}
    }
  \]
  and set by convention that $N^\rel_u(x) = -\infty$ if $x < 0$.
  Define also
  \[
    \eta_u(x)
    :=
    \eta_u\leftl(
      \frac{x-p\bigl(
        p-1-\gcd(p-1, \, u)
      \bigr)}{p-1}
    \rightr)
    =
    \frac{
      x
      - p(p-1)
      + \gcd(p-1, \, u)
    }{(p-1)^2}
    - \frac {pu}{p-1}
  \]
  and
  \[
    N_u(x) :=
    N_u^\rel\leftl(
      \frac{x-p\bigl(
        p-1-\gcd(p-1, \, u)
      \bigr)}{p-1}
    \rightr).
  \]
\end{definition}

We have
\[
  N_u(x) =
  \begin{cases}
    - \infty
    & \text{if }
    x < p(p-1)-p\gcd(p-1, \, u)
    \\
    \cardsuchthat{
      i \in \Z
    }{
      p \nmid i \\
      -\eta_u(x) \leq i < \frac{p u}{p-1}
    }
    \in \Z_{\geq 0}
    & \text{otherwise}.
  \end{cases}
\]

\begin{lemma}
  \label{lem:local:bound-disc-bound-val}
  Fix an element~$a \in K^\times$, let $u := v(a)$, and fix an element~$g \in K$ satisfying $v(g) \geq \frac{pu}{p-1}$ or $p \nmid v(g)$ (e.g., $g \in \cF_a$).
  Then, for any $x \geq 0$, we have the equivalences
  \begin{align*}
    \delta^{(a)}(g) \leq x
    &\iff
    v(g) \geq
    \min\leftl(
      \frac{pu}{p-1}, \,
      -\eta^\rel_u(x)
    \rightr)
    \\
    \delta(\rho_{g,a}) \leq x
    &\iff
    v(g) \geq
    \min\leftl(
      \frac{pu}{p-1}, \,
      -\eta_u(x)
    \rightr)
  \end{align*}
\end{lemma}

\begin{proof}
  By \iref{thm:local-disc}{thm:local-disc:formula-bis}, we have $\delta^{(a)}(g) \leq x$ if and only if either $v(g) \geq \frac{pu}{p-1}$ or
  \begin{align*}
    & pu+(p-1)(1-v(g))-\delta(\chi_a) \leq x
    \\ \iff &
    pu+(p-1)-(p-1)v(g)-(p-1-\gcd(p-1, \, u)) \leq x
    \\ \iff &
    v(g) \geq
    \frac {pu}{p-1}
    - \frac{x-\gcd(p-1, \, u)}{p-1}.
  \end{align*}
  The first claim follows.
  The second claim follows from this, cf.~\Cref{def:disc-rel,def:local-reldisc-tasty}.
\end{proof}

\begin{lemma}
  \label{lem:count-hom-count-inertype}
  For any~$x \in \R_{\geq0}$ and for any~$a \in K^\times$, we have
  \[
    \cardsuchthat
      {\rho \in \Hom(\Gamma_K, C_p\rtimes C_{p-1})}
      {
        \pi \circ \rho = \chi_a \\
        \delta^\rel(\rho) \leq x
      }
    =
    p \cdot
    \cardsuchthat
      {g\in\cF'_a}
      {\delta^{(a)}(g) \leq x}.
  \]
\end{lemma}

\begin{proof}
  By the observations of \Cref{subsn:fiberwise}, we have
  \[
    \cardsuchthat
      {\rho \in \Hom(\Gamma_K, C_p\rtimes C_{p-1})}
      {
        \pi \circ \rho = \chi_a \\
        \delta^\rel(\rho) \leq x
      }
    =
    \begin{cases}
      p \cdot
      \cardsuchthat
        {[g] \in K/\wpa(K)}
        {\delta^{(a)}(g) \leq x}
      &
      \text{if } a \not\in (K^\times)^{p-1},
      \\
      \cardsuchthat
        {[g] \in K/\wpa(K)}
        {\delta^{(a)}(g) \leq x}
      &
      \text{if } a \in (K^\times)^{p-1}
    \end{cases}
  \]
  If $a \not\in (K^\times)^{p-1}$, then $K/\wpa(K) \simeq \cF_a = \cF'_a$, and the claim follows.
  Assume now that $a \in (K^\times)^{p-1}$.
  By \iref{thm:local-fundom}{thm:local-fundom:is-fundom}, we have
  \[
    K/\wpa(K)
    \simeq
    \Bigl(
      \bigslant\kappa{\wps A (\kappa)}
    \Bigr)
    \pi^{pk}
    \oplus \cF_a'
  \]
  and $\card{\kappa/\wps A (\kappa)} = p$ by \Cref{eq:size-wp-finite}, proving the claim.
\end{proof}

\begin{proposition}[Fiberwise count]
  \label{prop:local-fiberwise-count}
  For any $a \in K^\times$ and for any $x \in \R$, letting $u := v(a)$ and using the convention that~$\card\kappa^{-\infty} = 0$, we have
  \begin{align*}
    \cardsuchthat{
      \rho \in \Hom(\Gamma_K, C_p \rtimes C_{p-1})
    }{
      \pi\circ\rho = \chi_a \\
      \delta^\rel(\rho) \leq x
    }
    & =
    p \card\kappa^{N^\rel_u(x)}
    \\
    \cardsuchthat{
      \rho \in \Hom(\Gamma_K, C_p \rtimes C_{p-1})
    }{
      \pi\circ\rho = \chi_a \\
      \delta(\rho) \leq x
    }
    & =
    p \card\kappa^{N_u(x)}.
  \end{align*}
\end{proposition}

\begin{proof}
  If $x \geq 0$, we have
  \begin{align*}
    \cardsuchthat{
      \rho \in \Hom(\Gamma_K, C_p \rtimes C_{p-1})
    }{
      \pi\circ\rho = \chi_a \\
      \delta^\rel(\rho) \leq x
    }
    & =
    p \cdot
    \cardsuchthat
      {g\in\cF'_a}
      {\delta^{(a)}(g) \leq x}
    \tag*{by \Cref{lem:count-hom-count-inertype}}
    \\
    & =
    p \cdot
    \cardsuchthat
      {g\in\cF'_a}
      {v(g) \geq \min(\frac{pu}{p-1}, -\eta^\rel_u(x))}
    \tag*{by \Cref{lem:local:bound-disc-bound-val}}
    \\
    & =
    p
    \card\kappa^{N^\rel_u(x)}
    \tag*{by definition of~$\cF'_a$.}
  \end{align*}
  If $x<0$, the conventions~$N^\rel_u(x) = -\infty$ and~$\card\kappa^{-\infty}=0$ ensure that both sides are zero, so the formula still holds.

  The second claim easily follows as we have the following equivalence, for any~$x \in \R$ and for any~$\rho$ such that~$\pi\circ\rho=\chi_a$ (cf.~\Cref{def:disc-rel}, and compare with the definition of~$N_u(x)$):
  \[
    \delta(\rho) \leq x
    \iff
    \delta^\rel(\rho)
    \leq
    \frac{x-p\delta(\chi_a)}{p-1}
    = \frac{x-p\bigl(p-1-\gcd(p-1, \, u)\bigr)}{p-1}.
    \qedhere
  \]
\end{proof}

\begin{remark}
  The formulas in \Cref{prop:local-fiberwise-count} depend on~$a$ only in the form of its valuation~$u = v(a)$, i.e., the distribution of the discriminants of the local $(C_p \rtimes C_{p-1})$-extensions lifting a given Kummer $C_{p-1}$-extension only depends on the inertial type of that Kummer extension.
  This is analogous to the phenomenon described in \cite[Theorem~1.3]{wildlocglob}.
\end{remark}

\subsection{Total local counts.}

Summing over all possible characters~$\chi \in \Hom(\Gamma_K, C_{p-1})$ yields an exact formula for the total count of $(C_p \rtimes C_{p-1})$-extensions of the local field $K = \kappa\ppar\pi$ of discriminant exponent $\leq x$, for any $x \in \R_{\geq0}$.
To illustrate the generality of the claims and methods, we include counts for the surjective homomorphisms (corresponding to Galois field extensions~$L|K$ with an isomorphism $\Gal(L|K) \simeq C_p \rtimes C_{p-1}$) and dihedral $D_p$-extensions (with $D_p = C_p\rtimes C_2 \simeq C_p \rtimes \frac{p-1}2C_{p-1}$) when $p \neq 2$:

\begin{theorem}
  \label{thm:local-total-count}
  For any $x \in \R_{\geq0}$, we have the total count
  \[
    \frac{1}{p(p-1)}
    \cardsuchthat{
      \rho \in \Hom(\Gamma_K, C_p \rtimes C_{p-1})
    }{
      \delta(\rho) \leq x
    }
    =
    \sum_{u = 0}^{p-2}\,
      \card\kappa^{N_u(x)}
  \]
  and, for $p \neq 2$, we have the following counts for surjective (resp.~dihedral) extensions:
  \begin{align*}
    \frac{1}{p(p-1)}
    \cardsuchthat{
      \rho \in \Hom^\surj(\Gamma_K, C_p \rtimes C_{p-1})
    }{
      \delta(\rho) \leq x
    }
    & =
    \sum_{u = 0}^{p-2}\,
      \frac{\varphi(\gcd(u,p-1))}{\gcd(u,p-1)}
      \Bigl(
        \card\kappa^{N_u(x)}
        - 1
      \Bigr)
    \tag*{if $p\neq2$}
    \\
    \frac{1}{2p}
    \cardsuchthat{
      \rho \in \Hom^\surj(\Gamma_K, D_p)
    }{
      \delta(\rho) \leq x
    }
    & =
    \frac12
    \card\kappa^{N_0(\frac{p-1}2x)}
    +
    \card\kappa^{N_{\frac{p-1}2}(\frac{p-1}2x)}
    -\frac32
    \tag*{if $p\neq2$}
  \end{align*}
  where~$\varphi$ denotes Euler's totient function.
\end{theorem}


\begin{proof}
  We first deal with the first claim, counting all extensions.
  We have
  \begin{align*}
    &\cardsuchthat{
      \rho \in \Hom(\Gamma_K, C_p \rtimes C_{p-1})
    }{
      \delta(\rho) \leq x
    }
    \\
    ={}&
    \sum_{\chi \in \Hom(\Gamma_K, C_{p-1})}\;
      \cardsuchthat{
        \rho \in \Hom(\Gamma_K, C_p \rtimes C_{p-1})
      }{
        \pi\circ\rho = \chi \\
        \delta(\rho) \leq x
      }
    \\
    ={}&
    \sum_{u = 0}^{p-2}\;
      \sum_{[\tilde a] \in \F_q^\times/(\F_q^\times)^{p-1}}\;
        \cardsuchthat{
          \rho \in \Hom(\Gamma_K, C_p \rtimes C_{p-1})
        }{
          \pi\circ\rho = \chi_{\tilde a \pi^u} \\
          \delta(\rho) \leq x
        }
    &
    \tag*{by \Cref{eq:kummer,prop:kummer-local-normal-form}}
    \\
    ={}&
    \sum_{u = 0}^{p-2}\;
      \sum_{[\tilde a] \in \F_q^\times/(\F_q^\times)^{p-1}}\,
        p \card\kappa^{N_u(x)}
    &
    \tag*{by \Cref{prop:local-fiberwise-count}}
    \\
    ={}&
    p(p-1)
    \cdot
    \sum_{u = 0}^{p-2}\,
      \card\kappa^{N_u(x)}
    \tag*{by \Cref{rk:kummer-ext-finite-fields}.}
  \end{align*}

  Assume that $p \neq 2$.
  We now explain how to adapt the argument above to count only surjections.
  First, the bottom $C_{p-1}$-extension~$\chi \in \Hom(\Gamma_K, C_{p-1})$ must be surjective.
  For any fixed~$u := \inv\chi \in \{0, \ldots, p-2\}$, via the bijection of \Cref{prop:kummer-inertia}, this amounts to picking an element~$\lambda \in C_{p-1}$ such that $\gcd(u,\lambda,p-1)=1$: the number of such~$\lambda$ is the number of $\lambda \in \{0, \ldots, p-2\}$ that are coprime to $\gcd(u,\,p-1)$, i.e., this is $\frac{p-1}{\gcd(u,\,p-1)}$ times more than~$\varphi(\gcd(u,\,p-1))$:
  \[
    \cardsuchthat{\lambda \in C_{p-1}}{\langle u, \lambda \rangle = C_{p-1}}
    =
    \frac{p-1}{\gcd(u,\,p-1)}\varphi(\gcd(u,\,p-1)).
  \]
  Consequently, in the computation above, instead of multiplying by $\card{\F_q^\times/(\F_q^\times)^{p-1}} = p-1$ in the last step, we should only multiply each summand (for a given~$u$) by the quantity above.
  Then, we must also ensure that the $C_p$-extension on top be of degree~$p$, i.e., non-trivial (since~$p$ is prime, there are no other options)---once we have already made sure that the Kummer $C_{p-1}$-extension~$\chi_a$ is surjective, then~$\chi_a$ is non-trivial as $p\neq2$, so we are in the case ``$a \not\in (K^\times)^{p-1}$'' in which $\cF_a \simeq \cF'_a$.
  Therefore, excluding the trivial $C_p$-extension is the same as excluding the trivial element of~$\cF'_a$.%
  \footnote{
    What we are saying is that this $C_p$-extension (above a ramified $C_{p-1}$-extension, such that the full extension is a $(C_p \rtimes C_{p-1})$-extension) is trivial if and only if it is unramified; this can be shown directly using the fact that the inertia subgroup must be normal in the Galois group when $u \neq 0$, and the fact that the residue field extension must be cyclic when $u=0$.
  }
  Reproducing the proof of \Cref{prop:local-fiberwise-count}, we obtain $p(\card\kappa^{N_u(x)}-1)$ choices for~$\rho$ in the fiber above a fixed $\chi \in \Hom^\surj(\Gamma_K, C_{p-1})$.
  With these changes, the computation above becomes
  \begin{align*}
    \cardsuchthat{
      \rho \in \Hom^\surj(\Gamma_K, C_p \rtimes C_{p-1})
    }{
      \delta(\rho) \leq x
    }
    & =
    \sum_{u = 0}^{p-2}\,
      \frac{p-1}{\gcd(u,\,p-1)}
      \varphi(\gcd(u,\,p-1))
      \cdot
      p
      (\card\kappa^{N_u(x)}-1)
    \\
    & =
    p(p-1)
    \cdot
    \sum_{u = 0}^{p-2}\,
      \frac{\varphi(\gcd(u,\,p-1))}{\gcd(u,\,p-1)}
      (\card\kappa^{N_u(x)}-1)
  \end{align*}
  and the second claim follows.
  The case of dihedral extensions is treated similarly: the Kummer $C_{p-1}$-extension must have its image equal to~$\frac{p-1}2 C_{p-1} \simeq C_2$, so either $u = \frac{p-1}2$ and then~$\lambda$ is picked freely in $\{0, \frac{p-1}2\}$ ($2$ choices), or $u=0$ and $\lambda = \frac{p-1}2$.
  As above, the condition that the $C_p$-extension on top be surjective amounts to it being non-trivial, which amounts to it being unramified.
  Note also that the discriminant exponent of~$\rho$ as a $D_p$-extension is $\frac{\card{D_p}}{\card{C_p \rtimes C_{p-1}}} = \frac2{p-1}$ times its discriminant exponent as a $(C_p \rtimes C_{p-1})$-extension.
  Therefore, we get $p(\card\kappa^{N_u(\frac{p-1}2x)}-1)$ choices in each fiber.
  The total count is then
  \[
    \cardsuchthat{
      \rho \in \Hom^\surj(\Gamma_K, D_p)
    }{
      \delta(\rho) \leq x
    }
    =
    \underbrace{
      p(\card\kappa^{N_0(\frac{p-1}2x)}-1)
    }_{\text{case } u=0, \, \lambda=\frac{p-1}2}
    +
    \underbrace{
      2p(\card\kappa^{N_{\frac{p-1}2}(\frac{p-1}2x)}-1)
    }_{\text{case } u=\frac{p-1}2, \, \lambda\in\{0,\frac{p-1}2\}}.
    \qedhere
  \]
\end{proof}

\subsection{The behavior of~$N_u(x)$}
\label{subsn:behavior-nux}

As shown in \Cref{thm:local-total-count}, the number of $(C_p\rtimes C_{p-1})$-extensions of~$\F_q\ppar\pi$ with discriminant exponent~$\leq x$ is controlled by the numbers~$N_u(x)$ for $u \in \{0, \ldots, p-2\}$.
We now study these numbers.

\subsubsection{Estimates}
We begin with a computation of~$N_u(x)$:

\begin{lemma}
  \label{lem:Nu}
  For any $u \in \Z$ and any $x \in \R$, we have
  \begin{align*}
    N^\rel_u(x)
    & =
    -\infty
    & \text{if }
    x < 0
    \\
    N^\rel_u(x)
    & =
    0
    & \text{if }
    0 \leq x \leq \gcd(p-1, \, u)
    \\
    N^\rel_u(x)
    & =
    u + \floor{\eta^\rel_u(x)} - \floor{\frac{\eta^\rel_u(x)}p}
    & \text{if }
    x  \geq \gcd(p-1, \, u)
    \\
    N_u(x)
    & =
    -\infty
    & \text{if }
    x < p(p-1)-p\gcd(p-1, \, u)
    \\
    N_u(x)
    & =
    0
    & \text{if }
    p(p-1)-p\gcd(p-1, \, u) \leq x \leq p(p-1)-\gcd(p-1, \, u)
    \\
    N_u(x)
    & = u + \floor{\eta_u(x)} - \floor{\frac{\eta_u(x)}p}
    & \text{if }
    x \geq p(p-1)-\gcd(p-1, \, u).
  \end{align*}
\end{lemma}

\begin{proof}
  When $x < 0$, we have $N^\rel_u(x) = - \infty$ by convention.
  When $0 \leq x \leq \gcd(p-1, \, u)$, we have $-\eta^\rel_u(x) \geq \frac{pu}{p-1}$, so $N^\rel_u(x) = 0$.
  Assume now that $x \geq \gcd(p-1, \, u)$, so $-\eta^\rel_u(x) \leq \frac{pu}{p-1}$.
  For any real numbers $a \leq b$, we have
  \[
    \cardsuchthat{i \in \Z}{p \nmid i \\ a \leq i < b}
    =
    \cardsuchthat{i \in \Z}{a \leq i < b} - \cardsuchthat{i \in \Z}{\frac ap \leq i < \frac bp}
    =
    \ceil{b}-\ceil{a}-\ceil{\frac bp}+\ceil{\frac ap}.
  \]
  Applying this to~$N^\rel_u(x)$ gives the result, as $\ceil{\frac{pu}{p-1}} =  \ceil{u + \frac{u}{p-1}} = u +  \ceil{\frac{u}{p-1}}$.
  The claims for~$N_u(x)$ follow from those for~$N^\rel_u(x)$.
\end{proof}

A consequence is the ``periodic'' behavior of~$N_u(x)$: for any $x  \geq p(p-1)-\gcd(p-1, \, u)$ and for any~$\ell \in \Z_{\geq0}$, we have $\eta_u(x+p(p-1)^2\ell) = \eta_u(x) + p\ell$, so
\[
  N_u(x+ p(p-1)^2 \ell)
  =
  u + \floor{\eta_u(x)+p\ell} - \floor{\frac{\eta_u(x)}p+\ell}
  =
  N_u(x) + (p-1) \ell.
\]

In \Cref{cor:values-nux}, we shall estimate~$N_u(x)$ more precisely using the following observation:

\begin{lemma}
  \label{lem:estimate-floor-minus-floor}
  For all~$x \in \R$, we have
  \[
    \left\lvert
      \floor x -
      \floor{\frac xp} -
      \leftl(
        1-\frac1p
      \rightr)
      x
    \right\rvert
    \leq
    1-\frac1p
  \]
\end{lemma}

\begin{proof}
  The left-hand side is $p$-periodic and the right-hand side is constant, so we can assume that $0 \leq x < p$, in which case $\floor{\frac xp} = 0$.
  Write $x = k+r$ with $k \in \{0, \ldots, p-1\}$ and $0 \leq r < 1$.
  We have
  \[
    \floor x -
    \leftl(
      1-\frac1p
    \rightr)
    x
    =
    \frac kp - \leftl(1-\frac1p\rightr) r
  \]
  which is larger than $\frac0p - \leftl(1-\frac1p\rightr) 1 = 1-\frac1p$, and at most equal to $\frac{p-1}p- \leftl(1-\frac1p\rightr) 0 = 1-\frac1p$.
\end{proof}

\begin{corollary}
  \label{cor:values-nux}
  For any $u \in \{0, \ldots, p-2\}$ and any $x \geq p(p-1)-\gcd(p-1,u)$, we have
  \[
    N_u(x)
    \begin{cases}
      = y + 1 & \text{if } u=0 \text{ and } p(p-1)^2 \mid x \\
      = y & \text{if } p(p-1) \mid x \text{ and } (u \neq 0 \text{ or } p(p-1)^2 \nmid x) \\
      \in \{\floor y, \ceil y\} & \text{otherwise}
    \end{cases}
    \qquad\text{where }
    y = \frac x{p(p-1)} - 1.
  \]
  In particular, for a fixed~$x$, $N_u(x)$ takes at most two values (differing by~$1$) when~$u$ varies.
\end{corollary}

\begin{proof}
  By \Cref{lem:Nu} and \Cref{lem:estimate-floor-minus-floor}, for some  $\varepsilon \in [-1+\frac1p, 1-\frac1p]$, we have
  \begin{align*}
    N_u(x)
    & =
    u
    + \leftl(1-\frac1p\rightr) \eta_u(x)
    + \varepsilon
    \\ & =
    u
    +
    \frac{p-1}p
    \leftl(
      \frac{
        x-p(p-1)+\gcd(p-1, \, u)
      }{(p-1)^2}
      -
      \frac{pu}{p-1}
    \rightr)
    + \varepsilon
    \\ & =
    \frac{x}{p(p-1)}
    -1
    +
    \underbrace{
      \frac{\gcd(p-1, \, u)}{p(p-1)}
    }_{\in (0, \frac1p]}
    +
    \underbrace{
      \varepsilon
    }_{\in [-1+\frac1p,1-\frac1p]}
    =
    y
    + \underbrace{\varepsilon'}_{\in (-1,1]}.
  \end{align*}
  Since~$N_u(x)$ is an integer, this implies that~$N_u(x) = \floor y + 1$ (if $0 < \varepsilon' \leq 1$) or that $N_u(x) = \floor y$ (if $-1 < \varepsilon' \leq 0$).
  If~$y$ is not an integer, then $\floor y + 1 = \ceil y$ and we are done, so assume that $p(p-1) \mid x$.
  If $u \neq 0$, then $\frac{\gcd(p-1, \, u)}{p(p-1)} < \frac 1p$ so $\varepsilon' < 1$ and~$N_u(x) = y+1$ is not possible so $N_u(x)=y$ and we are done, so assume that $u=0$.
  We have $x = p (p-1)(y+1)$ and thus
  \[
    \eta_u(x)
    =
    \frac{p(p-1)(y+1)-p(p-1)+p-1}{(p-1)^2}
    =
    \frac{p(y+1)}{p-1} - 1
    =
    y + \frac{y+1}{p-1}
  \]
  so
  \[
    N_u(x)
    =
    y
    +
    \floor{\frac{y+1}{p-1}}
    -
    \floor{\frac{y+1}{p-1} - \frac1p}
    =
    \begin{cases}
      y + 1 & \text{if there is an integer in } (\frac{y+1}{p-1}-\frac1p, \frac{y+1}{p-1}] \\
      y &  \text{otherwise.}
    \end{cases}
  \]
  If there is an integer in the interval~$(\frac{y+1}{p-1}-\frac1p, \frac{y+1}{p-1}]$, there is in particular an integer in the interval~$(\frac y{p-1}, \frac{y+1}{p-1}]$, which means that $p-1 \mid y + 1$, i.e., $p(p-1)^2 \mid x$.
  Conversely, if $p(p-1)^2 \mid x$, then the interval $(\frac{y+1}{p-1}-\frac1p, \frac{y+1}{p-1}]$ contains the integer $\frac{y+1}{p-1} = \frac{x}{p(p-1)^2}$, so $N_u(x) = y + 1$.
\end{proof}

\Cref{thm:local-total-count,cor:values-nux} imply the following: if~$p(p-1)^2 \mid x > 0$, then
\[
  \frac{1}{p(p-1)}
  \cardsuchthat{
    \rho \in \Hom(\Gamma_K, C_p \rtimes C_{p-1})
  }{
    \delta(\rho) \leq x
  }
  =
  \underbrace{
    \card\kappa^{\frac{x}{p(p-1)}}
  }_{u = 0}
  +
  \underbrace{
    (p-2) \card\kappa^{\frac{x}{p(p-1)}-1}
  }_{u\neq0}
\]
and if~$p(p-1) \mid x$ but~$p(p-1)^2 \nmid x$ then
\[
  \frac{1}{p(p-1)}
  \cardsuchthat{
    \rho \in \Hom(\Gamma_K, C_p \rtimes C_{p-1})
  }{
    \delta(\rho) \leq x
  }
  =
  (p-1) \card\kappa^{\frac{x}{p(p-1)}-1}.
\]
In general, whenever~$p(p-1)^2 \nmid x \geq p(p-1)$, we have
\[
  (p-1) \card\kappa^{\floor{\frac{x}{p(p-1)}-1}}
  \leq
  \frac{1}{p(p-1)}
  \cardsuchthat{
    \rho \in \Hom(\Gamma_K, C_p \rtimes C_{p-1})
  }{
    \delta(\rho) \leq x
  }
  \leq (p-1) \card\kappa^{\ceil{\frac{x}{p(p-1)}-1}}
\]
and similar estimates are obtained for surjective homomorphisms.
For dihedral extensions, the term~$\frac{x}{p(p-1)}$ in the exponents is to be replaced by~$\frac x{2p}$.

\subsubsection{Visualizations}

We are going to draw visualizations of ~$N_u(x)$ for $0 \leq u \leq p-2$ (rows) and $p(p-1)^2 \leq x < 2p(p-1)^2$ (columns), i.e., over an entire period.
The topmost row corresponds to unramified Kummer extensions ($u=0$).
In these visualizations, darker shades of gray will indicate smaller values (the black on the left is~$\floor{\frac{p(p-1)^2}{p-1}-1} = p-2$, the white on the right is~$(p-2)+(p-1) = 2p-3$).

In each individual row (for a fixed~$u$), the value of~$N_u(x)$ increases~$p-1$ times over a period: in general, the consecutive ``jumps'' are evenly spaced (at distance~$(p-1)^2$), corresponding to the jumps of~$\floor{\eta_u(x)}$ when~$\eta_u(x)$ is an integer, except for two consecutive jumps that are twice as far (at distance~$2(p-1)^2$) because of the negating effect of the term~$-\floor{\eta_u(x)/p}$ when~$\eta_u(x)$ is an integer multiple of~$p$.

\begin{figure}[H]
  \setlength{\fboxsep}{0pt}%
  \fbox{\includegraphics[scale=0.07]{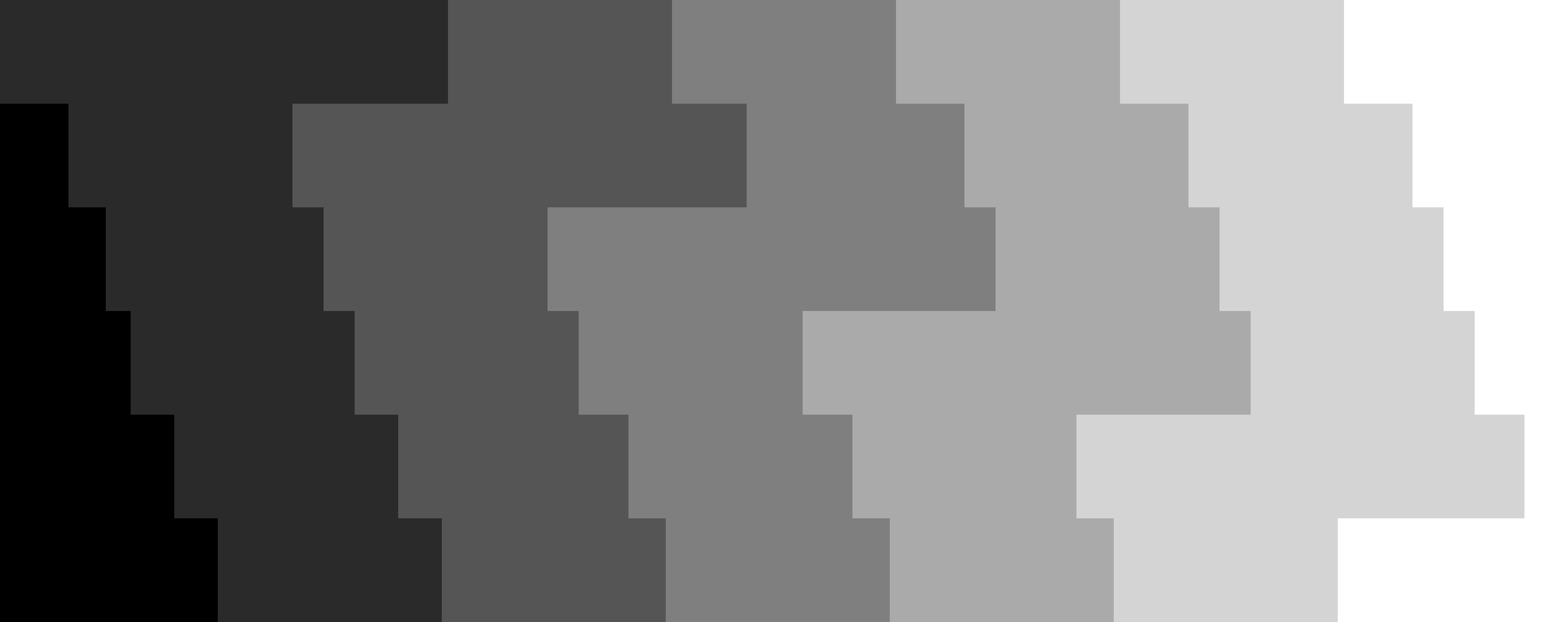}}
  \caption{
    A visualization of~$N_u(x)$ in the case~$p=7$.
    The values range from~$5$ to~$11$.
  }
  \label{fig:Npux7}
\end{figure}

We can visualize the claims of \Cref{cor:values-nux} on \Cref{fig:Npux7}.
Every column contains at most two shades of gray.
In the leftmost column (multiples of~$p(p-1)^2$), the first row ($u=0$) is strictly brighter than the others.
We see $6=p-1$ regularly-spaced vertical stripes in which a single shade is represented across an entire column (multiples of~$p(p-1)$ that are not multiples of~$p(p-1)^2$).

\begin{figure}[H]
  \setlength{\fboxsep}{0pt}%
  \fbox{\includegraphics[scale=0.07]{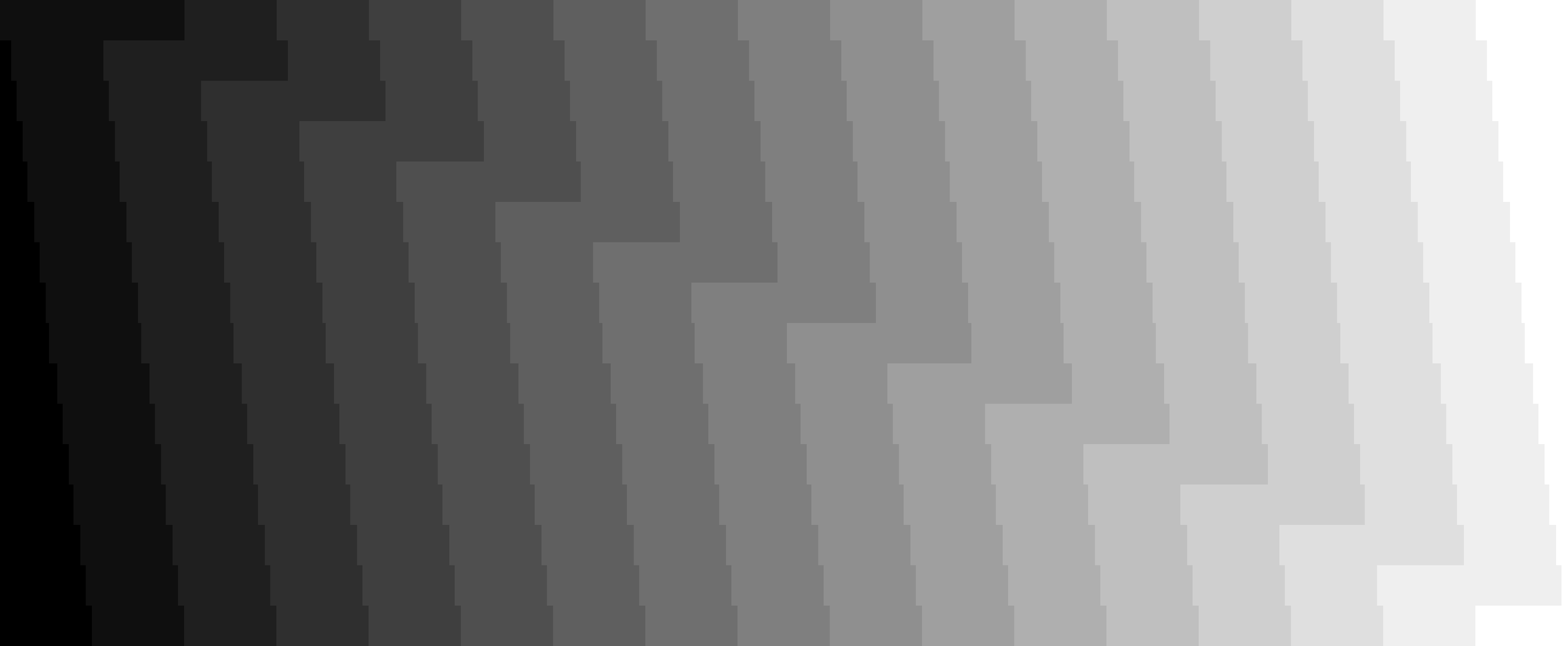}}
  \caption{
    A similar visualization for~$p=17$.
    The values range between~$15$ and~$31$.
  }
\end{figure}

\begin{example}[$\fS_3$-extensions in characteristic~$3$]
  \label{ex:s3-char3}
  Assume that~$p=3$.
  Then, from the corresponding visualization
  \begin{center}
    \includegraphics[scale=0.07]{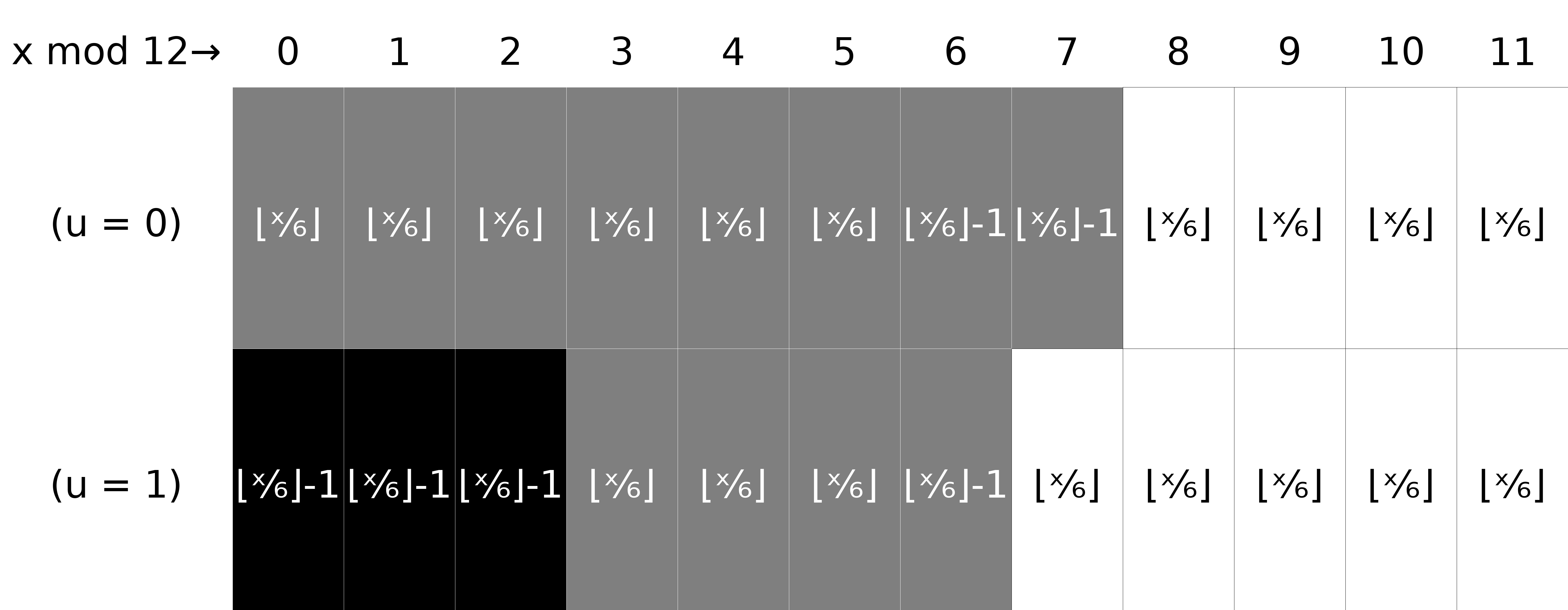}
  \end{center}
  and from \Cref{thm:local-total-count}, we deduce the following complete answer for the symmetric group~$\fS_3 = C_3 \rtimes C_2$:
  \[
    \frac{1}{6}
    \cardsuchthat{
      \rho \in \Hom(\Gamma_K, \fS_3)
    }{
      \delta(\rho) \leq x
    }
    =
    \begin{cases}
      1
      & \text{if } 0 \leq x \leq 2
      \\
      2
      & \text{if } 3 \leq x \leq 6
      \\
      \card\kappa^{\floor{\frac x6}} + 
      \card\kappa^{\floor{\frac x6}-1}
      & \text{if }
      x \geq 6 \text{ and }
      (x \bmod 12) \in \{0,1,2,7\}
      \\
      2\card\kappa^{\floor{\frac x6}}
      & \text{if }
      x \geq 6 \text{ and }
      (x \bmod 12) \in \{3,4,5,8,9,10,11\}
      \\
      2 \card\kappa^{\frac x6 - 1}
      & \text{if }
      x \geq 6 \text{ and } 
      x \equiv 6 \bmod 12.
    \end{cases}
  \]
  In terms of exactly prescribed discriminants (instead of upper bounds), this becomes
  \[
    \frac{1}{6}
    \cardsuchthat{
      \rho \in \Hom(\Gamma_K, \fS_3)
    }{
      \delta(\rho) = x
    }
    =
    \begin{cases}
      1
      & \text{if } x = 0 \text{ or } x = 3
      \\
      0
      & \text{if } (x \bmod 12) \in \{1,2,4,5,6,9,10,11\}
      \\
      \card\kappa^{\floor{\frac x6}} - 
      \card\kappa^{\floor{\frac x6}-1}
      & \text{if }
      x \geq 6 \text{ and }
      (x \bmod 12) \in \{0,3,7,8\}.
    \end{cases}
  \]
  The congruence classes $0,3,7,8 \bmod 12$ correspond to the indices of the columns in which some cell is brighter than the cell on its left.
  (When it happens that several cells, say~$r$, get brighter in the same column, then for the discriminant exponents~$x$ in the corresponding congruence class modulo~$p(p-1)^2$ we obtain $r\Bigl(\card\kappa^{\floor{\frac x{p(p-1)}}} - \card\kappa^{\floor{\frac x{p(p-1)}}-1}\Bigr)$ instead.)

  We deduce from this an exact expression for the generating function for the discriminant exponents of $\fS_3$-extensions of $K = \kappa\ppar\pi$:
  \[
    \frac16
    \sum_{\rho \in \Hom(\Gamma_K, \fS_3)}
      X^{\delta(\rho)}
    =
    1
    +
    X^3
    +
    \frac
    {
      (\card\kappa-1)
      (X^7 + X^8 + \card\kappa X^{12} + \card\kappa X^{15})
    }{
      1-X^{12}\card\kappa^2
    }.
  \]
\end{example}

\renewcommand{\addcontentsline}[3]{}
\emergencystretch=1em
\bibliographystyle{alphaurl}
\bibliography{refs.bib}

\end{document}